\documentclass[a4paper,12pt]{article}

\usepackage{amsmath,amsthm,amssymb,amsfonts}
\usepackage{graphics,graphicx,color}

\graphicspath{{/EPSF/}{Figures/}}

\makeatletter

\@addtoreset{equation}{section}
\makeatother

\newtheorem{theorem}{Theorem}[section]
\newtheorem{lemma}[theorem]{Lemma}
\newtheorem{proposition}[theorem]{Proposition}
\newtheorem{corollary}[theorem]{Corollary}

\newtheorem{remark}[theorem]{Remark}

\def\ep{\varepsilon}

\def\C{\mathbb C}
\def\R{\mathbb R}

\def\N{\mathbb N}
\def\pa{\partial}
\def\b{\backslash}

\def\V{\mathrm V}

\def\diam{{\rm diam}(X)}

\def \Rm {\mathbb R}

\def \Sm {\mathbb S}
\newcommand{\eps}{\varepsilon}

\newcommand{\dint}{\displaystyle\int}

\newcommand{\cout}[1]{}

\def\ep{\varepsilon}

\def\C{\mathbb C}
\def\R{\mathbb R}

\def\N{\mathbb N}
\def\pa{\partial}
\def\b{\backslash}

\def\V{\mathrm V}

\def\r{\rangle}
\def\l{\langle}

\title{Generalized stability estimates in inverse transport theory}
\author{Guillaume Bal\thanks{Department of Applied Physics and 
        Applied Mathematics, Columbia University, New York NY, 10027,
gb2030@columbia.edu} \and Alexandre Jollivet\thanks{Laboratoire de Math\'ematiques Paul Painlev\'e,
CNRS UMR 8524/Universit\'e Lille 1 Sciences et Technologies,
59655 Villeneuve d'Ascq Cedex, France,
alexandre.jollivet@math.univ-lille1.fr}}

\begin{document}

\maketitle

\begin{abstract}
  Inverse transport theory concerns the reconstruction of the absorption and scattering coefficients in a transport equation from knowledge of the albedo operator, which models all possible boundary measurements. Uniqueness and stability results are well known and are typically obtained for errors of the albedo operator measured in the $L^1$ sense. We claim that such error estimates are not always very informative. For instance, arbitrarily small blurring and misalignment of detectors result in $O(1)$  errors of the albedo operator and hence  in $O(1)$ error predictions on the reconstruction of the coefficients, which are not useful.
  
This paper revisit such stability estimates by introducing a more forgiving metric on the measurements errors, namely the $1-$Wasserstein distances, which penalize blurring or misalignment by an amount proportional to the width of the blurring kernel or to the amount of misalignment. We obtain new stability estimates in this setting.

 We also consider the effect of errors, still measured in the $1-$ Wasserstein distance, on the generation of the probing source. This models blurring and misalignment  in  the design of (laser) probes and allow us to consider a discretized sources. Under appropriate assumptions on the coefficients, we quantify the effect of such errors on the reconstructions. 
  
\end{abstract}

{\bf keywords:} Linear transport; inverse problems; stability estimates; Wasserstein distance.

\section{Introduction}
\label{sec:intro}

We consider the inverse problem theory of the following stationary (time-independent) linear transport equation:
\begin{equation}
  \label{eq:steady}
  \begin{array}{ll}
  v\cdot\nabla u(x,v) + \sigma(x,v) u = \dint_{V} k(x,v',v) u(x,v') dv',\quad&
  (x,v)\in X\times V\\
  u(x,v) =g(x,v),\quad &(x,v)\in \Gamma_-.
  \end{array}
\end{equation}
The solution $u(x,v)$ models the density of particles, such as
photons, as a function of space $x$ and velocity $v$.  The spatial
domain $X$ is here a convex, bounded, open subset of $\Rm^n$ for dimension
$n\geq2$, with a $C^1$ boundary $\partial X$. The space of velocities
$V=\Sm^{n-1}$ is chosen to be the unit sphere to simplify the presentation.
The extinction coefficient $\sigma(x,v)$, referred to as {\em absorption} coefficient, models particle loss due to scattering or intrinsic absorption at $(x,v)$ while the {\em scattering} coefficient $k(x,v',v)$ models scattering from direction $v'$ into direction $v$ at position $x$.

The sets of incoming conditions $\Gamma_-$ and outgoing conditions
$\Gamma_+$ are defined by
\begin{equation}
  \label{eq:Gammas}
  \Gamma_\pm = \{(x,v) \in \partial X\times V, \mbox{ s.t. } 
    \pm v\cdot \nu(x) > 0 \},
\end{equation}
where $\nu(x)$ is the outward normal vector to $X$ at $x\in\partial X$. 

Under a sub-criticality condition on $(\sigma,k)$, we can show that the above equation admits a unique solution $u(x,v)$ with a well-defined trace on $\Gamma_+$. The albedo operator, mapping $g$ on $\Gamma_-$ to $u_{|\Gamma_+}$ is denoted by $A$ so that $u_{|\Gamma_+}=Ag$.

The most general way to probe the domain $X$ is to assume knowledge of the above albedo operator $A$. Indeed, any boundary measurement may be modeled as measuring the effect on $\Gamma_+$ of a probing signal emitted on $\Gamma_-$. There is a significant literature on the analysis of such an inverse problem in the setting where $A$ is known as an operator from $L^1(\Gamma_-,d\xi)$ to $L^1(\Gamma_+,d\xi)$ with $d\xi$ an appropriate weighted version of the Lebesgue measure on $\Gamma_\pm$. It is for instance known that $A$ uniquely and stably determines $(\sigma(x),k(x,v',v))$. Obstructions to the reconstruction of $(\sigma(x,v),k(x,v',v))$ are also well understood \cite{McDST-IP-10}. 

The main objective of this paper is to revisit the stability estimates available in the literature and obtain results for measurement errors that are more realistic experimentally. Imposing error estimates in ${\cal L}(L^1(\Gamma_-,d\xi),L^1(\Gamma_+,d\xi))$ is indeed very constraining. Let us consider for instance a signal $g$ supported in the $\eps$ vicinity of a point $(x_0,v_0)\in\Gamma_-$ for some $\eps\ll1$. In the absence of scattering ($k\equiv0$), the solution $u$ is then supported in the vicinity of the line segment $(x_0,x_1)$ with $x_1-x_0$ collinear to $v_0$ and $(x_1,v_0)\in\Gamma_+$. Now let us assume that the measurements involve an amount of blurring equal to $\eta\ll1$. In other words, the solution $u_{|\Gamma_+}$ is replaced by the convolution $\phi_\eta * u_{|\Gamma_+}$, with $\phi_\eta$ an approximation of identity with support of width $\eta$ (in each dimension). It is not difficult to convince oneself that $\|u_{|\Gamma_+}- \phi_\eta * u_{|\Gamma_+}\|_{L^1(\Gamma_+,d\xi)}$ will be of order $O(1)$ independent of $\eta$ as soon as $\eps$ is significantly smaller than $\eta$. In other words, the error between $A$ and $\phi_\eta *A$ (with obvious notation) will be of order $O(1)$ independent of $\eta$ (arbitrarily small positive).

Any inevitable, even arbitrarily small, blurring or misalignment of the detectors will result in an error of order $O(1)$ when errors on $A$ are measured in ${\cal L}(L^1(\Gamma_-,d\xi),L^1(\Gamma_+,d\xi))$. The same error holds for imperfect sources, whose support is typically not of width $\eps$ for arbitrarily small $\eps$ and whose alignment may not be known with infinite precision as required by the above metric imposed on $A$. Standard error estimates then imply that the errors on the reconstruction on $(\sigma,k)$ are of order $O(1)$ as well, which is not very informative.

In this paper, we introduce measurement error on the albedo operator $A$ that are significantly more forgiving and more likely to be met in practical settings. Our main objective is to show how the reconstruction of $(\sigma,k)$ degrades as more general measurement errors are considered.

Heuristically, we wish to consider a metric in which small blurring and small misalignment result in small errors. As we have just seen, $L^p$-based metrics are not adapted. Nor is the Radon norm on bounded measures. We want a metric in which the distance between $\delta(x-a)$ and $\delta(x-b)$ is small when $b-a$ is small. A classical metric that is quite adapted to such a problem and whose manipulation is reasonably straightforward is the $1-$Wasserstein metric.

We introduce the following family of Wasserstein distances:
\begin{equation}
W_{1,\kappa}(\mu,\nu)=\sup_{\|\phi\|_\infty\le 1,\ \textrm{Lip}(\phi)\le \kappa}\l\phi,\mu-\nu\r,
\end{equation}
for $\kappa$ a positive constant. The Wasserstein distance is typically introduced to measure the distance between probability measures $\mu$ and $\nu$ so that $\l C,\mu-\nu\r=0$ for any constant $C$. The above bound on $\phi$ is then not necessary and $\kappa$, which scales linearly, is typically chosen equal to $1$. In our applications, the measures of the form $u_{|\Gamma_+}$ are not normalized and the above uniform bound on $\phi$, normalized to $1$, is necessary to form a metric.

We verify that for each $\kappa>0$, the above object is indeed a metric on the space of bounded measures and satisfies the typical properties of the Wasserstein metric on probability measures (such as, e.g., metrizing weak convergence of bounded measures); see, e.g., \cite{MR2459454}.

The role of $\kappa$ allows us to introduce a parameter that measures confidence in the available measurements. Note that $W_{1,\kappa}(\delta_a,\delta_b)=\kappa|b-a|$. A small error with large $\kappa$ therefore means that we are quite confident in our measurement setting to distinguish between points $a$ and $b$.

Our first objective is to assess the effects of measurement errors of $Ag$ in the above Wasserstein metric on the reconstruction of the coefficients $(\sigma,k)$. As we mentioned above, this allows us to model a wide class of measurement errors. We verify that the Wasserstein distance between a measure and its blurring by a kernel $\phi_\eta$ is of order $\eta$. Errors in the position of detectors, whose locations are typically not known to arbitrary accuracy, are handled similarly. As we shall see, the same uniqueness results hold as in the setting of errors on $A$ in the $L^1$ sense. However, compared to the latter case, stability estimates degrade. The reconstructions of appropriate functionals of $\sigma$ and $k$ are now obtained with an error that is H\"older and no longer Lipschitz. This means that errors $\eps$ on $A$ in the Wasserstein sense translate into errors of order $\eps^\beta$ for $\beta<1$ on the coefficient reconstructions, whereas $\beta=1$ holds when errors on $A$ are measured in the $L^1$ sense.

A second objective is to consider the setting in which the source term $g$ is replaced by another source term that is close in the same Wasserstein sense. This allows us to consider the case of errors in the alignment and design of (e.g.,) laser probes. This type of error, typically neglected in standard stability estimates, is important in practice to model probing signals that are never perfectly known. Such errors also allow us to model a discrete probing source configuration. Indeed, let $g$ be a non-negative source with a support in the $h$ vicinity of $(x_0,v_0)\in\Gamma_-$ and integrating to $1$. Then we verify that $W_{1,\kappa}(g,\delta_{x_0}(x)\delta_{v_0}(v))\leq \kappa h$. Any source term may thus be approximated by a superposition of delta functions on a grid with an error proportional to the mesh size $h$. In such a setting, we allow the replacement of a line integral by the integral along a nearby line. It is thus clear that reconstructions may be accurate only when the absorption and scattering coefficients are sufficiently smooth. With this additional assumption, we obtain stability estimates in this setting.

\medskip

What the optimal choice should be to quantify measurement and probing errors in an inverse problem is no doubt very subjective. Our choice of the $1-$Wasserstein distance in inverse transport theory is based on two observations: (i) it is sufficiently versatile to provide small penalties on a large class of classical errors such as small blurring or misalignment, both at the detector and source levels; and (ii)  its structure is quite amenable to relatively simple mathematical analysis and generalizations of results obtained in the setting of $L^1$ errors. 

\medskip
 
The inverse transport setting and the main results of the paper are summarized in section \ref{sec:main}.  The errors in the probing sources require that we analyze the propagation of Lipschitz constants by a forward transport problem. Such a forward analysis is carried out in section \ref{sec:fwd}. The stability estimates are based on careful analyses of multiple scattering contributions to the albedo operator, which are presented in section \ref{sec:estms}. The detailed  proofs of the results are given in sections \ref{sec:inv} and \ref{sec_app}.

The methodology followed here draws on the large body of works on stationary inverse transport theory as developed in, e.g., \cite{CS-CPDE-96,CS-OSAKA-99,W-AIHP-99,BJ-IPI-08,ST-PAMS-09,McDST-IP-10}. For concreteness, we concentrate here on the stationary inverse problem with full angular measurements. Our results, which primarily follow from the detailed analysis of the transport Green's kernel as developed in, e.g., \cite{BJ-IP-09,BJJ-IP-10}, (most likely) adapt to other settings, such as the settings of time-dependent and/or angularly averaged measurements; see \cite{BJ-SIMA-10,BJ-HGS-09,BJ-IP-09,BJLM-CPDE-11,BLM-IPI-08,L-IP-08,LM-CPDE-09,SU-MAA-03} for additional references as well as \cite{B-IP-09} for a recent review.

Inverse transport theory has been widely studied besides the stability estimates mentioned in the above works. For a selected list of references, see, e.g., \cite{AKP-VSP-02,BT-SIMA-07,Bond-DAN-92,Bond-JMAA98,larsen84,larsen88,MC-TTSP86,MC-NSE92,S-IO-03,SU-MAA-03,SU-APDE-08,tamasan-IP02}.

\section{Inverse transport theory and main results}
\label{sec:main}

Let us first summarize our main assumptions on the coefficients and the geometry that allows us to define the albedo operator $A$. 

We define the times of escape of free-moving particles from $X$ as
\begin{equation}\label{eq:taupm}
  \tau_\pm(x,v) = \inf \{s>0 | x\pm sv \not\in X\}
\end{equation}
and $\tau(x,v)=\tau_+(x,v)+\tau_-(x,v)$. On the boundary sets
$\Gamma_\pm$, we introduce the measure
$d\xi(x,v)=|v\cdot\nu(x)|d\mu(x)dv$, where $d\mu(x)$ is the surface
measure on $\partial X$.

We say that the optical parameters $(\sigma,k)$ are admissible when
\begin{equation}
  \label{eq:admiss}
  \begin{array}{l}
   0\leq \sigma \in L^\infty(X\times V)\\
   0\leq k(x,v',\cdot) \in L^1(V) \mbox{ a.e. in } X\times V\\
   \sigma_p(x,v') := \dint_V k(x,v',v)dv \in L^\infty(X\times V).
  \end{array}
\end{equation}
We say that the problem is sub-critical when either one of the following assumptions holds
\begin{align}
  \tag{HSC1}\label{eq:sc1}
  \sigma-\sigma_p \geq 0 &\quad (x,v) \mbox{ a.e. } \\
  \tag{HSC2}\label{eq:sc2}
   \|\tau \sigma_p\|_\infty <1.
\end{align}

We now define the following Banach space
\begin{equation}
  \label{eq:Wsteady}
  W := \big\{ u\in L^1(X\times V) | v\cdot\nabla_x u\in L^1(X\times V),
    \tau^{-1} u\in L^1(X\times V) \big\},
\end{equation}
with its natural norm. 

Under either \eqref{eq:sc1} or \eqref{eq:sc2}, we obtain for admissible parameters that the albedo operator $A$ is well posed from $u_{|\Gamma_-}=g\in L^1(\Gamma_-,d\xi)$ to $Ag=u_{|\Gamma_+}\in L^1(\Gamma_+,d\xi)$, where $u$ is the unique solution to \eqref{eq:steady} in $W$; see e.g. \cite[Theorem 2.2]{B-IP-09} and references cited there.

\medskip

We are now in a position to consider the inverse transport problem, which aims to describe what may or may not be reconstructed in $(\sigma,k)$ and with which type of stability. Before we describe our measurement setting, we need to decompose the albedo operator into components with different singular structures.

We first introduce the lifting operator
\begin{equation}
  \label{eq:Isteady}
  J g(x,v) = \exp{\Big(-\dint_0^{\tau_-(x,v)}\sigma(x-sv,v)ds\Big)}
  g(x-\tau_-(x,v)v,v).
\end{equation}
It is proved in \cite{CS-OSAKA-99} that $J$ is a bounded
operator from $L^1(\Gamma_-,d\xi)$ to $W$. Let us next define the bounded operator
\begin{align}
{\footnotesize K u(x,v) = \dint_0^{\tau_-(x,v)}  \hspace{-.75cm}
  \exp\Big(-\int_0^t \sigma(x-sv,v)ds\Big)\dint_V
   k(x-tv,v',v)u(x-tv,v')dv' dt }
  \label{eq:K}
\end{align}
for $(x,v)\in X\times V$. Note that \eqref{eq:steady} may then be recast as
$(I-K) u = J g$.

With this notation, we decompose the albedo operator as
\begin{equation}
  \label{eq:decalb}
   \begin{array}{rclclcl}
  A g &=& J g \big|_{\Gamma_+}
    &+& K J g \big|_{\Gamma_+}
    &+&  K^2 (I-K)^{-1}J g \big|_{\Gamma_+}\\
   &:= & B g &+&S g &+&Mg.
   \end{array}
\end{equation}
Physically, the above decomposition separates the albedo operator into the {\em ballistic} component $Bg$, the {\em single scattering} component $Sg$, and the rest. In dimension $n\geq3$, the former two components are {\em more singular} than $Mg$ when $g$ approximates a point source. We also introduce $S_c=S+M$ the scattering component so that $A=B+S_c$.

\medskip

We are now ready to present our measurement setting. Consider two sets of coefficients $(\sigma_j,k_j)$ for $j=1,2$ and define $A_j$ as the corresponding albedo operator. Our objective is to analyze how errors on measurements of $A_j$ translate into errors on the reconstruction of (functionals of) $(\sigma_j,k_j)$. 

Let $g$ be a given (exact) source in $L^1(\Gamma_-,d\xi)$ normalized to $\|g\|_{L^1(\Gamma_-,d\xi)}=1$. We consider two approximations $g_j$ of $g$, which are used to  probe the domain $X$. The resulting solutions on $\Gamma_+$ are thus given by $A_j g_j$. Let $\phi$ be a (test) function in $L^\infty(\Gamma_+)$ such that $\|\phi\|_{L^\infty(\Gamma_+)}\leq1$. Then we may define the measurement (``projected" onto that test function)
\begin{equation}
\label{eq:mj}
   m \equiv m_{g,\phi}= \eps_j+ \l \phi,A_jg_j \r,\qquad j=1,2,
\end{equation}
where $\eps_j$ is a measurement error for such a projection. 

We assume the measurement to be given experimentally, i.e., independent of the configuration $j$. Taking the difference of the above relations for $j=1$ and $j=2$ yields
\begin{equation}
   \l (A_1-A_2)g,\phi \r = \eps_2-\eps_1 + \l A_2(g_2-g)-A_1(g_1-g),\phi \r.\label{eq:depart}
\end{equation}
Appropriate (properly normalized) choices for $g\in L^1(\Gamma_-)$ and $\phi\in L^\infty(\Gamma_+)$ provide information about the coefficients $(\sigma_j,k_j)$. Note that the left-hand side is bounded by the error of the albedo operator $A_1-A_2$ in the $L^1$ sense and is the starting point of all current stability estimates in inverse transport theory; see, e.g., \cite{B-IP-09}.

Our objective is to generalize such estimates to the setting where $g_j$ is an approximation of $g$ in the Wasserstein sense, and where errors between $A_1$ and $A_2$ are measured in the same sense. This implies that the test function $\phi$ can no longer be considered as a general bounded function but instead is a function bounded by $1$ and with Lipschitz constant bounded by $\kappa$. 

The error estimates proceed as follows. We first extract information from the ballistic part of the albedo operator:
\begin{displaymath}
  \l (B_1-B_2)g,\phi \r = \eps_2-\eps_1 + \l A_2(g_2-g)-A_1(g_1-g)+(S_{c2}-S_{c1})g,\phi \r,
\end{displaymath}
and next from the single scattering component, making sure that the ballistic contribution vanishes $\l B_j g,\phi\r=0$,
\begin{displaymath}
  \l (S_1-S_2)g,\phi \r =  \eps_2-\eps_1 + \l A_2(g_2-g)-A_1(g_1-g)+(M_{2}-M_{1})g,\phi \r.
\end{displaymath}

To explicit the information we obtain from the above decomposition, we need to introduce the
  following notation. For $x$ and $y$ in $\bar X$, we define
  \begin{equation}
  \label{eq:Exy}
  E(x,y) := \exp\Big(-\int_0^{|x-y|}\hspace{-.25cm}
   \sigma\big(x-s\frac{x-y}{|x-y|},\frac{x-y}{|x-y|}\big)ds \Big).
\end{equation}
Let us also define for $(x,v,w)\in X\times\Sm^{n-1}\times\Sm^{n-1}$:
\begin{equation}
  \label{eq:Eplus}
  \tilde E (x,v,w) = \exp\Big( -\int_0^{\tau_-(x,v)} \hspace{-.25cm}\sigma(x-sv,v)ds
    - \int_0^{\tau_+(x,w)} \hspace{-.25cm}\sigma(x+sw,w)ds \Big),
\end{equation}
the total attenuation factor on the broken line
$(x-\tau_-(x,v)v,x,x+\tau_+(x,w)w)$. Let us introduce the error term $\eps=|\eps_1|+|\eps_2|$. 

We are now ready to state our main stability results. The first one assumes that $g_j=g$, i.e., that the probing source is exact. Since measurements errors are given in a Wasserstein sense, we assume the additional $W^{1,\infty}$ regularity constraints on the coefficients and assume that $L={\rm max}({\rm Lip}(\sigma),{\rm Lip}(k))$ is bounded and $k$ vanishes in the vicinity of the boundary $\pa X$ (see \eqref{INV} and \eqref{HL} for the precise assumption). To avoid uninteresting singularities arising when $\kappa\to0$, which corresponds to infinitely blurring detectors, we assume that the fixed constant $\kappa\geq1$. 

In this setting, we have
\begin{theorem}
\label{thm:exactsource}
Let $(x_0,v_0)\in\Gamma_-$ and $y_0=x_0+\tau_+(x_0,v_0)v_0$,. Then for a constant $C$ that depends on $\|k_j\|_{W^{1,\infty}}$, $\|\sigma_j\|_{W^{1,\infty}}$ and subcritical conditions, we obtain that
\begin{equation}
\label{eq:stab1}
 |E_1(x_0,y_0)-E_2(x_0,y_0)| \leq  C\Big(\Big(\dfrac \eps \kappa \Big)^{\frac{n-1}n}\vee \eps\Big).
\end{equation}
In dimension $n=3$, we have
\begin{equation}
 \dint_V\dint_0^{\tau_+(x_0,v_0)}\hspace{-.75cm}
   \big| \tilde E_1k_1 - \tilde E_2k_2\big|
  (x_0+sv_0,v_0,v) ds dv \leq  C \Big(\Big(\dfrac \eps \kappa \Big)^{\frac12}(1+\sqrt{|\ln({\ep\over\kappa})|})\vee \eps\Big).\label{eq:stab2a}
\end{equation}
In dimension $n\geq4$, we have 
\begin{equation}
 \dint_V\dint_0^{\tau_+(x_0,v_0)}\hspace{-.75cm}
   \big| \tilde E_1k_1 - \tilde E_2k_2\big|
  (x_0+sv_0,v_0,v) ds dv \leq  C\Big(\Big(\dfrac \eps \kappa \Big)^{\frac12}\vee \eps\Big).\label{eq:stab2b}
\end{equation}
\end{theorem}
Here, $a\vee b={\rm max}(a,b)$. 
Note the similarity with the theorems obtained in the $L^1$ setting; see \cite{B-IP-09}. The main difference is that Lipschitz stability estimates are typically replaced by less favorable H\"older estimates. This is expected since our setting is precisely aimed to incorporate a larger class of model and measurement errors. Note also that the errors improve as $\kappa$ increases, i.e., as our confidence in our detectors to  discriminate nearby phase-space points increases. If $\kappa$ scales appropriately with $\eps$, then Lipschitz stability is even restored. For instance, in \eqref{eq:stab1} on the stability of line integrals of $\sigma(x,v)$, we observe that a Lipschitz estimate (an error proportional to $\eps$) is obtained for $\kappa$ on the order of $\eps^{-\frac 1{n-1}}$. Note that according to these estimates, there is no gain in increasing the resolution capabilities of the detector, i.e., increasing $\kappa$ beyond this order.

Let us now consider the more constraining setting of approximate probing sources.
As we indicated above, replacing $g$ by an approximation $g_j$ in the Wasserstein sense includes replacing the ``real" source by an approximation supported on a (discrete) grid. As a consequence, we do not expect to reconstruct parameters that oscillate rapidly at the grid level. We thus need
to impose additional $W^{1,\infty}$ regularity constraints on the coefficients and assume that $L={\rm max}({\rm Lip}(\sigma),{\rm Lip}(k))$ is bounded (see \eqref{INV}, \eqref{HL} for the precise assumptions) and we need to impose an additional constraint on the boundary $\pa X$ (see \eqref{HL1}). Let $\kappa\ge 1$ and let us introduce
\begin{equation}
\label{eq:errorsource}
\delta_j= \sup_{g\in L^1(\Gamma_-)} W_{1,\kappa}(g,g_j)
\end{equation}
 the maximal error that is made in the experimental setting to approximate the probe $g$. Let also $\delta=\delta_1+\delta_2$. Then we have the following result:
 \begin{theorem}
\label{thm:appsource}
Under the aforementioned assumptions, we obtain using the above notation that
\begin{equation}
\label{eq:stab3}
|E_1(x_0,y_0)-E_2(x_0,y_0)| \leq  C\Big(\Big(\dfrac {\eps+\delta} \kappa \Big)^{\frac{n-1}n}\vee (\eps+\delta)\Big).
\end{equation}
Similarly, we obtain that 
in dimension $n=3$
\begin{gather}
\label{eq:stab4a}
\dint_V\dint_0^{\tau_+(x_0,v_0)}\hspace{-.75cm}
   \big| \tilde E_1k_1 - \tilde E_2k_2\big|
  (x_0+sv_0,v_0,v) ds dv  \\ \leq  C \Big(\Big(\Big(\dfrac {\eps+\delta} \kappa \Big)^{\frac12}(1+\sqrt{|\ln({\ep+\delta\over\kappa})|})\Big)\vee (\eps+\delta)\Big), \nonumber
\end{gather}
and in dimension $n\geq4$
\begin{equation}
\label{eq:stab4b}
\dint_V\dint_0^{\tau_+(x_0,v_0)}\hspace{-.75cm}
   \big| \tilde E_1k_1 - \tilde E_2k_2\big|
  (x_0+sv_0,v_0,v) ds dv \leq  C \Big(\Big(\dfrac {\eps+\delta} \kappa \Big)^{\frac12}\vee (\eps+\delta)\Big).
\end{equation}
\end{theorem}
In  the above theorem, the constant $C$ in front of $\delta$ depends on the sub-criticality assumption; see the proofs below.
   
\medskip
 
It is now classical to deduce stability results on the coefficients $(\sigma,k)$ from the stability results on the above functionals $E_j$ and $\tilde E_jk_j$  of $(\sigma,k)$ \cite{B-IP-09}. Let us consider for concreteness the setting developed in \cite{ST-PAMS-09}, which provides a simple setting to present stability estimates. 
  
Inspection of \eqref{eq:stab1} or \eqref{eq:stab3} shows that the integral of $\sigma(x,v)$ between any two points at the domain's boundary is stably reconstructed (and exactly when $\eps=\delta=0$). In other words, we have access to the Radon transform of $\sigma$.
When $\sigma=\sigma(x,v)$ and $k=0$, it is clearly not possible to uniquely reconstruct $\sigma$ from its line integrals. The obstruction to unique reconstruction was analyzed in \cite{ST-PAMS-09}.
  $(\sigma,k)$ and $(\sigma',k')$ are called gauge equivalent if there is a function $\phi(x,v)$ such that $0<\phi_0\leq \phi(x,v)\leq\phi_0^{-1}$, $|v\cdot\nabla \phi|$ bounded, $\phi=1$ on $\partial X\times V$, and
  \begin{displaymath}
  \sigma' = \sigma -v\cdot\nabla \log \phi,\qquad k'(x,v',v)=\dfrac{\phi(x,v)}{\phi(x,v')} k(x,v',v).
\end{displaymath}
Note that $u$ and $\phi u$ then solve the same transport equation and satisfy the same boundary conditions. Let $<\sigma,k>$ be the class of equivalence. What \cite{ST-PAMS-09} show is that the class of equivalence is uniquely and stably characterized by the left-hand side in \eqref{eq:stab1}-\eqref{eq:stab2b}. 

Let ${\mathfrak e}_b(\eps,\delta)$ be the right-hand side in either \eqref{eq:stab1} or \eqref{eq:stab3} and ${\mathfrak e}_s(\eps,\delta)$ be the right-hand side in either \eqref{eq:stab2a}--\eqref{eq:stab2b} or \eqref{eq:stab4a}--\eqref{eq:stab4b} depending on the context of interest. Then following  \cite{ST-PAMS-09}, we have
\begin{theorem}
\label{thm:stabcoefs}
Under the hypotheses of either Theorem \ref{thm:exactsource} or \ref{thm:appsource}, there is a constant $C$ such that
\begin{equation}
\label{eq:stabcoefs}
\|\tau(\tilde\sigma_1-\tilde\sigma_2)\|_\infty \leq C {\mathfrak e}_b(\eps,\delta)\quad \mbox{ and } \quad \|\tilde k_1-\tilde k_2\|_1 \leq C ({\mathfrak e}_b(\eps,\delta)+{\mathfrak e}_s(\eps,\delta)),
\end{equation}
for appropriate elements in the class of equivalence $<\tilde\sigma_j,\tilde k_j>=<\sigma_j,k_j>$, where $\tau$ is defined below \eqref{eq:taupm}.
\end{theorem}

  
The remarks stated after Theorem \ref{thm:exactsource} still hold for the estimates provided in Theorems \ref{thm:appsource} and \ref{thm:stabcoefs}. Let us conclude this section by a few remarks.

The estimates in Theorem \ref{thm:stabcoefs} should be interpreted with some care. The stability estimates \eqref{eq:stab1} and \eqref{eq:stab3} provide an estimate on the reconstruction of the {\em line integral}, i.e., the X-ray transform, of $\sigma$, not of $\sigma$ itself. When $\sigma=\sigma(x)$ is independent of the variable $v$, such an information is sufficient to uniquely determine $\sigma$ (by inverse Radon transform \cite{natterer-86}); but {\em not} in the $L^\infty$ sense. The error in \eqref{eq:stabcoefs} holds for (one element in) a class of equivalence, not for every element of that class of equivalence. For instance, when $\sigma=\sigma(x)$, we may construct an element $\tilde \sigma(x,v)$ in the same class of equivalence by assigning to $\tilde \sigma(x,v)$ the line integral of $\sigma(x)$ of direction $v$ passing through $x$ divided by the length of that segment in $X$ (so that the line integrals of $\sigma(x)$ and $\tilde\sigma(x,v)$ agree). Then $\tilde \sigma(x,v)$ is clearly reconstructed with stability in the $L^\infty$ sense, but not $\sigma(x)$ itself. The reconstruction of $\sigma$ then requires inverting a Radon transform, which is a classical problem that is not stable in the $L^\infty$ sense. The same remark holds for the reconstruction of the scattering coefficient $k(x,v',v)$.

The estimates in the preceding three theorems are written for functionals of the coefficients $(\sigma,k)$ and, in the case of  Theorem \ref{thm:stabcoefs}, for classes of equivalence. How such estimates can be converted to estimates on the reconstruction of $(\sigma=\sigma(x),k=k(x,v',v))$ has already been addressed in the literature. We refer the reader to, e.g., \cite{B-IP-09,ST-PAMS-09,W-AIHP-99} for such results.

As we indicated in the introduction, stability estimates for measurement errors in the ${\cal L}(L^1(\Gamma_-,d\xi),L^1(\Gamma_+,d\xi))$ sense have been obtained in several works, for instance \cite{B-IP-09,BJ-IPI-08,ST-PAMS-09,W-AIHP-99}. Our work  \cite{BJ-IPI-08} introduces estimates on the geometry of multiple scattering contributions to the albedo operator that are more precise than is necessary for stability estimates in the ${\cal L}(L^1(\Gamma_-,d\xi),L^1(\Gamma_+,d\xi))$ sense. Such more precise estimates were used in \cite{BJ-HGS-09} as a first attempt to consider more general errors than those that can be analyzed in the ${\cal L}(L^1(\Gamma_-,d\xi),L^1(\Gamma_+,d\xi))$ setting. In the current paper, these estimates on multiple scattering contributions constitute the central tool in the derivation of the results in the above theorems.  The rest of the paper is essentially devoted to the detailed, albeit lengthy, derivation of these estimates.

 


%

\section{Forward Theory}
\label{sec:fwd}

In this section, we introduce the notation on forward transport equations that we follow throughout the paper and recall the relevant results on the forward transport theory in the $L^1$ setting. We next present the results we need in the dual setting of bounded functions. Finally, we consider the setting of transport solutions that are Lipschitz continuous. Such an analysis is necessary as we aim to model measurement errors in the Wasserstein distance.

\subsection{Recall on the $L^1$ framework.}
Using the notation introduced in the preceding section, we make the following assumption:
\begin{eqnarray}
\tau\sigma\in L^\infty(X\times V) \quad \mbox{ and } \quad  \tau\sigma_p(x,v)\in L^\infty(X\times V).\label{Hyp}
\end{eqnarray}
The bounded domain $X$ is assumed to be of class $C^1$.
We introduce the Banach space $W$ which is the completion of $C^1(\bar X\times V)$ with the norm
\begin{equation}
\|u\|_W:=\|\tau^{-1}u\|_{L^1}+\|v\cdot\nabla_x u\|_{L^1}.\label{g1}
\end{equation}
We have $W=\{u:X\times V\to \C\textrm{ measurable }\ |\ (\tau^{-1}u, v\cdot \nabla_x u)\in L^1(X\times V)^2\}$.
The following trace properties hold.

\begin{lemma}[see \cite{CS-OSAKA-99}]
The following operators $J:L^1(\Gamma_-,d\xi)\to W$ and $j_\pm:W\to L^1(\Gamma_+,d\xi)$ are well defined and bounded
\begin{equation*}
J(u)(x,v)=e^{-\int_0^{\tau_-(x,v)}\sigma(x- sv,v)ds}u(x- \tau_-(x,v)v,v)
\end{equation*}
for a.e. $(x,v)\in X\times V$ and for $u\in L^1(\Gamma_-,d\xi)$, and
\begin{equation*}
j_\pm(u)(x,v)=u(x,v)\textrm{ for } (x,v)\in \Gamma_\pm,\ u\in C^1(\bar X\times V),
\end{equation*}
where $\|j_\pm(u)\|\le \|u\|_W$.
\end{lemma}
The linear Boltzmann equation is then recast as the following integral equation 
\begin{equation}
(I-K)u=J\phi,\label{Boltz}
\end{equation}
where $K$ is a bounded operator in $L^1(X\times V,\tau^{-1} dx dv)$ and 
\begin{equation}
K=TP,
\end{equation}
where $T$ and $P$  are bounded operators from $L^1(X\times V)$ to $W$ and from  $L^1(X\times V,\tau^{-1} dx dv)$ to $L^1(X\times V)$, respectively and are given by
\begin{eqnarray}
T\phi(x,v)&:=&\int_0^{\tau_-(x,v)}e^{-\int_0^t\sigma(x-sv,v)ds}\phi(x-tv,v)dt, \\
P\phi(x,v)&:=&\int_V k(x,v',v)\phi(x,v')dv',
\end{eqnarray}
for $\phi\in L^1(X\times V)$ and a.e. $(x,v)\in X\times V$.
The operator $K$ maps $L^1(X\times V,\tau^{-1} dx dv)$ to $W$. 
We also have  $j_-T=0$. We recall the following well known solvability condition for the transport equation.

\begin{lemma}[see for instance \cite{BJ-IPI-08}]
Assume that 
\begin{equation}
I-K\textrm{ is an invertible operator in }L^1(X\times V,\tau^{-1} dx dv).\label{INV}
\end{equation} 
Then $(I-K)^{-1}J$ is a bounded operator from $L^1(\Gamma_-,d\xi)$ to $W$, and the albedo operator $A=j_+(I-K)^{-1}J$ is well defined and bounded.
\end{lemma}

We will use the following duality property
\begin{eqnarray*}
&&\big(L^1(X\times V,\tau^{-1}dx dv)\big)^*=\tau^{-1}L^\infty(X\times V):=\{\tau^{-1} g\ |\ g\in L^\infty(X\times V)\},\\
&&P^*u(x,v)=\int_V k(x,v,v')u(x,v')dv'
\end{eqnarray*}
for $u\in \tau^{-1}L^\infty(X\times V)$ and for a.e. $(x,v)\in X\times V$.
The operator $P^*$ is the dual operator for $P$ and is bounded from $L^\infty(X\times V)$ to $\tau^{-1}L^\infty(X\times V)$.
We compute $T^*$ the dual operator of $T$ which is bounded from $\tau^{-1}L^\infty(X\times V)$ to $L^\infty(X\times V)$.
For any $u\in L^1(X\times V)$ and $\phi\in \tau^{-1}L^\infty(X\times V)$ we obtain by successive changes of variables
\begin{eqnarray*}
&&\int_{X\times V}\hspace{-.5cm} Tu(x,v) \phi(x,v)dx dv=\int_{X\times V}\int_0^{\tau_-(x,v)} \hspace{-.5cm} e^{-\int_0^t \sigma(x-sv,v)ds}u(x-tv,v)dt\phi(x,v)dxdv\\
&=&\int_{\Gamma_-}\int_0^{\tau_+(x,v)}\int_0^r e^{-\int_0^t \sigma(x+(r-s)v,v)ds}u(x+(r-t)v,v)dt\phi(x+rv,v)dr d\xi(x,v)\\
&=&\int_{\Gamma_-}\int_0^{\tau_+(x,v)}\int_0^r e^{-\int_t^r \sigma(x+sv,v)ds}u(x+tv,v)dt\phi(x+rv,v)dr d\xi(x,v)\\
&=&\int_{\Gamma_-}\int_0^{\tau_+(x,v)}u(x+tv,v)\int_t^{\tau_+(x,v)}e^{-\int_t^r \sigma(x+sv,v)ds}\phi(x+rv,v)dr dt d\xi(x,v)\\
&=&\int_{\Gamma_-}\int_0^{\tau_+(x,v)}u(x+tv,v)\int_0^{\tau_+(x+tv,v)}e^{-\int_0^r \sigma(x+(t+s)v,v)ds}\phi(x+tv+rv,v)dr dt d\xi(x,v).
\end{eqnarray*}
Hence 
\begin{eqnarray*}
\int_{X\times V}Tu(x,v) \phi(x,v)dx dv
=\int_{X\times V}u(x,v)\int_0^{\tau_+(x,v)} e^{-\int_0^r \sigma(x+sv,v)ds}\phi(x+rv,v)dr dx dv,
\end{eqnarray*}
and we obtain
\begin{eqnarray*}
T^*\phi(x,v)=\int_0^{\tau_+(x,v)} e^{-\int_0^r \sigma(x+sv,v)ds}\phi(x+rv,v)dr,\ \textrm{a.e. }(x,v)\in X\times V.
\end{eqnarray*}
The operator $K^*=P^*T^*$ is a bounded operator in $\tau^{-1}L^\infty(X\times V)$.

\subsection{Consequences on the backward equation.}
The backward equation is given formally by 
\begin{eqnarray}
&&v\cdot\nabla_x u-\sigma u=-\int_\V k(x,v,v')u(x,v')dv',\ (x,v)\in X\times V,\label{back}\\
&&u_{|\Gamma_+}=\phi.
\end{eqnarray}
The first-order PDE is understood in the weak sense when acting on functions of $C^1(X\times V)$ compactly supported in $X\times V$.

First, let us introduce the Banach space $W^\infty:=\{u\in L^\infty(X\times V)\ |\ \tau v\cdot \nabla_x u\in L^\infty(X\times V)\}$ endowed with the norm
\begin{equation}
\|u\|_{W^\infty}=\|u\|_{L^\infty}+\|\tau v\cdot \nabla_x u\|_{L^\infty}.
\end{equation} 
Then, the following trace properties hold (see also for instance \cite{ES-ApplAnal-93}): The operators $\tilde J$ and $\tilde j$ are well defined bounded operators, where
\begin{equation*}
\tilde J:L^\infty(\Gamma_+)\to W^\infty,\ \tilde J(u)(x,v)=e^{-\int_0^{\tau_+(x,v)}\sigma(x+ sv,v)ds}u(x+\tau_+(x,v)v,v),\label{def_J}
\end{equation*}
for $u\in L^\infty(\Gamma_+)$ and for a.e. $(x,v)\in \Gamma_+$,
and 
\begin{equation*}
\tilde j:W^\infty\to L^\infty(\Gamma_-),
\end{equation*}
\begin{equation}
\tilde j(u)(x,v):={1\over \tau_+(x,v)}\int_0^{\tau_+(x,v)}\hskip -1cm\big(u(x+ tv,v)-(\tau_+(x,v)-t)v\cdot\nabla_x u(x+ tv,v)\big)dt,\label{def_j}
\end{equation}
for $u\in W^\infty$ and for a.e $(x,v)\in \Gamma_-$.
We next introduce the bounded operator $\tilde K:L^\infty(X\times V) \to W^\infty$ defined by 
\begin{eqnarray}
\tilde K=T^*P^*.
\end{eqnarray}
Then the boundary value problem \eqref{back} is equivalent to the following integral equation
\begin{equation}
(I-\tilde K)u=\tilde J\phi.
\end{equation}

By duality, condition \eqref{INV} is equivalent to the invertibility of the bounded operator $I-K^*$ in $\tau^{-1}L^\infty(X\times V)$.
\begin{lemma}
\label{lem_back}
Under condition \eqref{INV}, the bounded operator $I-\tilde K$ in $L^\infty(X\times V)$ is invertible and its inverse is given by 
$$
(I-\tilde K)^{-1}=I+T^*(I-K^*)^{-1}P^*.
$$
In addition, $(I-\tilde K)^{-1}$ is bounded from $L^\infty(X\times V)$ to  $W^\infty$ and the backward albedo operator $A_{\rm back}=\tilde j(I-\tilde K)^{-1}\tilde J$ is a bounded operator from $L^\infty(\Gamma_+)$ to $L^\infty(\Gamma_-)$.
\end{lemma}


As an application of Green's formula, we have
\begin{lemma}
Assume condition \eqref{INV}. Then
for $f\in L^1(\Gamma_-,d\xi)$ and $\phi\in L^\infty(\Gamma_+)$, we have 
\begin{equation}
\int_{\Gamma_+}Af \phi d\xi=\int_{\Gamma_-}fA_{\rm back}\phi d\xi.
\end{equation}
\end{lemma}

\subsection{Lipschitz behavior of the backward transport solutions when the optical parameters are Lipschitz.}

To make the mathematical analysis reasonably straightforward, we now assume that:
\begin{eqnarray} 
&&X\ is\ convex,\nonumber\\ 
&&The\ functions\ \sigma\ and\ k \ are\ Lipschitz\ functions\nonumber\\
&&on\ \bar X\ and \ \bar X\times V\times V 
\ respectively,\label{HL} \\
&&and\ k\ vanishes\ in\ a\ neigborhood\ of\ \pa X\times V\times V
. \nonumber
\end{eqnarray}

Let us denote by ${\rm supp}_X k$ the support of the function $k$ in the $x$ variable, i.e. ${\rm supp}_X k:=\overline{\{y\in X\ |\ k(y,v,v')\not= 0\textrm{ for some }(v,v')\in V^2 \}}$.
Let us introduce a smooth compactly supported function $\rho\in C^\infty(\R^n)$ so that 
\begin{equation}
0\le\rho\le 1,\ \rho(x)=1\textrm{ for }x\in {\rm supp}_X k,\ {\rm supp}\rho\subset X.
\end{equation}

We will use the following property.
We denote by $C^\alpha(X\times V)$ the space of  H\"older continuous functions of order $\alpha$, $0<\alpha<1$, (uniformly) on $X\times V$.
\begin{lemma}
\label{lem_Ksmooth}
Under condition \eqref{HL},
the operator $\tilde K=P^*T^*$ is a bounded operator in ${\rm Lip}(X\times V)$, and 
the operator $\tilde K^2$ is a bounded operator from $L^\infty(X\times V)$ to $C^\alpha(X\times V)$ and from $C^\alpha(X\times V)$ to ${\rm Lip}(X\times V)$, $0<\alpha<1$. 
\end{lemma}

\begin{proof}[Proof of Lemma \ref{lem_Ksmooth}]
We still denote by $\rho$ the multiplication operator by the function $\rho$:
\begin{equation}
\rho(u)(x,v):=
\left\lbrace
\begin{array}{l}
\rho(x)u(x,v)\textrm{ when }(x,v)\in X\times V,\\
0\textrm{ otherwise}.
\end{array}
\right.
\end{equation}

Let us extend $k$ and $\sigma$ by $0$ outside $X\times V^2$ and $X\times V$, respectively, and still denote 
by $k$ and $\sigma$ these extensions. Then $\tilde K=P^*T^*$ can be written as 
\begin{equation}
\tilde K u(x,v)=\int_{V\times \R} \Psi(x,v,v',r)\rho(u)(x+rv,v')dr dv\label{D1}
\end{equation}
for $u\in L^\infty(X\times V)$ where
\begin{equation}
\Psi(x, v,v',r):= e^{-\int_0^r \sigma(x+sv,v)ds}k(x+rv,v,v'), \ \ (x,v,v',r)\in \R^n\times V^2\times \R.
\end{equation}
Under condition \eqref{HL}, $\Psi$ is a compactly supported Lipschitz function on  $\R^n\times V^2\times \R$ and
$\tilde K$ is a bounded operator in ${\rm Lip}(X\times V)$.

We recall the explicit expression of $P^*\tilde K=P^*T^*P^*$:
\begin{equation}
P^*\tilde Ku(x,w)=\int_V k(x,w,v)\int_0^{\tau_+(x,v)} e^{-\int_0^r \sigma(x+sv,v)ds}\int_V k(x+rv,v,v')\rho(u)(x+rv,v')dv' dr dv,
\end{equation}
for  $u\in L^\infty(X\times V)$ and $(x,v)\in X\times V$. Then, we make the change of variables $y=x+rv$, $dy=r^{n-1}dr dv$ and obtain that
\begin{equation}
P^*\tilde Ku(x,w)=\int_{\R^n\times V} {\Phi(x,y-x,\theta,w,v')_{|\theta={y-x\over |y-x|}}\over |x-y|^{n-1}}\rho(u)(y,v')dy dv',\label{D2}
\end{equation}
for $u\in L^\infty(X\times V)$, where
\begin{equation}
\Phi(x,y-x,\theta,w,v')=k(x,w,\theta)k(y,\theta,v')e^{-|y-x|\int_0^1 \sigma(x+\ep (y-x),\theta)ds},\  \label{D3}
\end{equation}
for $(x,y,\theta,w,v')\in \R^n\times \R^n\times V^3$.
Under condition \eqref{HL}, $\Phi$ is a compactly supported Lipschitz function on  $\R^n\times\R^n\times V^3$.
Then, $P^*T^*P^*$ is a bounded operator from $L^\infty(X\times V)$ to $C^\alpha(X\times V)$ (see, for example,  \cite[Theorem 5.25]{F-Tata-83}).
In addition, we have  $P^*T^*P^*$ is a bounded operator from $C^\alpha(X\times V)$ to ${\rm Lip}(X\times V)$ (see, for example, \cite[Chapter 9, Section 7]{MK-AV-86} on weakly singular integrals after approximating $\Phi$ by a sequence of $C^1$ smooth functions $\Phi_n$ that are compactly supported in $\R^n\times\R^n\times V^3$ with the following properties : $\lim_{n\to+\infty}\|\Phi-\Phi_n\|_{L^\infty(\R^n\times\R^n\times V^3)}=0$ and there exists a constant $C$ that does not depend on $\Phi$ so that
$\sup_{n\in \N}{\rm Lip}(\Phi_n)\le C({\rm Lip}\Phi+\|\Phi\|_{L^\infty(\R^n\times\R^n\times V^3)})$).

From \eqref{D2} it follows that $P^*\tilde K=\rho P^*\tilde K$ and that
\begin{equation}
\tilde K^2u(x,w)=\int_{\R}\Theta(x,t,w)\rho P^*\tilde Ku(x+tw,w)dt,\label{D4}
\end{equation}
where 
\begin{equation}
\Theta(x,t,w)=e^{-\int_0^t\sigma(x+sw,w)ds}\rho(x+tw).\label{D5}
\end{equation}
Under condition \eqref{HL}, $\Theta$ is a Lipschitz function on $X\times\R\times V$.
Therefore from \eqref{D4} and the properties of $P^*\tilde K$ given above, we obtain that $\tilde K^2$ is a bounded operator from $L^\infty(X\times V)$ to $C^\alpha(X\times V)$  and  from $C^\alpha(X\times V)$ to ${\rm Lip}(X\times V)$.
\end{proof}

\begin{corollary}
\label{cor_inv}
Under conditions \eqref{INV} and \eqref{HL}, $(I-\tilde K)$ is a bounded and invertible operator in the Banach space ${\rm Lip}(X\times V)$.
\end{corollary}

\begin{proof}[Proof of Corollary \ref{cor_inv}]
The boundedness of  $I-\tilde K$ in ${\rm Lip}(X\times V)$ follows from the first statement of Lemma \ref{lem_Ksmooth}.
It remains to prove that $I-\tilde K$ is invertible as a bounded operator in  ${\rm Lip}(X\times V)$. 
Under condition \eqref{INV}, we already know that $I-\tilde K$ is invertible as a bounded operator in  $L^\infty(X\times V)$.
Hence $I-\tilde K$ is still one-to-one when we restrict the operator to the subspace  ${\rm Lip}(X\times V)$ of $L^\infty(X\times V)$.
Therefore we just need to check that $I-\tilde K$ is onto as a bounded  operator in  ${\rm Lip}(X\times V)$.
Let $g\in {\rm Lip}(X\times V)$. Then by \eqref{INV} and Lemma \ref{lem_back} there exists $f\in L^\infty(X\times V)$ so that $f-\tilde Kf=g$. By Lemma \ref{lem_Ksmooth}, it follows that $\tilde K^4 f\in{\rm Lip}(X\times V)$ and $\tilde f=g+\tilde Kg+\tilde K^2g+\tilde K^3 g+\tilde K^4 f\in {\rm Lip}(X\times V)$.
\end{proof}

From now on we will always assume in this Section that
\begin{eqnarray}
&&X\ is\ a\ strictly\ convex\ (in\ the \ strong\ sense)\ bounded\nonumber\\
&& domain\ of\ class\ C^2.\label{HL1}
\end{eqnarray}

\begin{lemma}
\label{lem_restrict}
The operator $\tilde j$ is a bounded operator from ${\rm Lip}(X\times V)$ to ${\rm Lip}(\Gamma_-)$.
\end{lemma}

\begin{proof}[Proof of Lemma \ref{lem_restrict}]
From \eqref{def_j}, it follows that for $u\in {\rm Lip}(X\times V)$ we have $\tilde j(u)(x,v)=u(x,v)$ for $(x,v)\in \Gamma_+$
where $u$ is extended by continuity to $\bar X\times V$. Therefore, we obtain that
${\rm Lip}(\tilde j(u))\le {\rm Lip}(u)$.
\end{proof}

\begin{lemma}[see, for instance, \cite{DPSU-AdvMath-07}]
\label{ref_lip}
Under our geometric assumptions \eqref{HL1} on $X$, the functions $\tau_\pm$ are $C^1$-functions on $\overline{\Gamma_\mp}$ and on $X\times V$. 
\end{lemma}

\begin{proof}
Under \eqref{HL1} let $\chi\in C^2(\R^n)$ be a defining function for $X$, i.e. 
$\chi(x)=0$ and $\nabla \chi(x)\not=0$ when $x\in \pa X$, $\chi(x)<0$ when $x\in X$,
  $\chi(x)>0$ when $x\in \R^n\b \bar X$, and 
\begin{equation}
\textrm{the Hessian matrix }\textrm{Hess}\chi(x)\textrm{ is positive at any point }x\in \pa X.\label{a}
\end{equation}

Then, we prove that $\tau_\pm$ are $C^1$-functions on $X\times V$, resp. $\Gamma_\mp$, by applying the implicit function theorem on
$$
\chi(x\pm tv)=0,\ x\in X\times  V,\textrm{ resp. }\Gamma_\mp ,\ t>0.
$$
Then $\tau_\pm\in C^1(\overline{\Gamma_\mp})$ follows from \eqref{a} and the Taylor expansion
\begin{equation}
\chi(x+ tv)=t\nabla\chi(x)\cdot v+t^2\int_0^1(1-\ep)\textrm{Hess}\chi(x+\ep tv)(v,v)d\ep
\end{equation}
for $(x,v,t)\in \pa X\times V\times \R$ ($\chi(x)=0$).
\end{proof}

\begin{remark}
Note that for $X=\mathbb{D}$, then $\tau_+(x,v)=-2v\cdot x$ for $(x,v)\in \Gamma_-$, and $\tau_+(x,v)=-x\cdot v+\sqrt{1-|x|^2+(x\cdot v)^2}$ for $(x,v)\in X\times  V$. In addition, we have 
$$
\pa_{x_i}\tau_+(x,v)=-v_i+{-x_i+(x\cdot v)v_i\over \sqrt{1-|x|^2+(x\cdot v)^2}},\ i=1, 2,\ (x,v)\in X\times V,
$$
and $\pa_{x_i}\tau_+(x,v)$ is not even bounded on $X\times V$, and $\tau_+$ is not a Lipschitz function on $\bar X\times  V$.
In particular, 
$$
\tilde K \tilde J\phi(x,v)=\tau_+(x,v)|V|,\ (x,v)\in \Gamma_+
$$
for $\phi$ the constant valued function $1$ on $\Gamma_+$
when $\sigma\equiv 0$ on $X\times V$ and $k\equiv 1$ on $X\times V^2$ (which breaks assumption \eqref{HL}). Therefore when $k$ does not vanish at the boundary then $\tilde K$ does not necessarily map ${\rm Lip}(X\times V)$ to ${\rm Lip}(X\times V)$.
  
\end{remark}

\begin{lemma}
\label{lem_ext}
Under the assumptions \eqref{HL} and \eqref{HL1},
the operator $\rho\tilde J$ is bounded from ${\rm Lip}(\Gamma_+)$ to ${\rm Lip}(X\times V)$, and the operator $\tilde j(1-\rho)\tilde J$ is bounded from ${\rm Lip}(\Gamma_+)$
to ${\rm Lip}(\Gamma_-)$. In addition, $\tilde K (1-\rho)\tilde J\equiv0$.
\end{lemma}

\begin{proof}[Proof of Lemma \ref{lem_ext}]
Let $\phi\in {\rm Lip}(\Gamma_+)$. Then from \eqref{def_J}, it follows that
\begin{equation}
\rho\tilde J\phi(x,v)=\rho(x)e^{-\int_0^{\tau_+(x,v)}\sigma(x+ sv,v)ds}\phi(x+\tau_+(x,v)v,v),\label{J1a}
\end{equation}
\begin{equation}
(1-\rho)\tilde J\phi(x,v)=(1-\rho(x))e^{-\int_0^{\tau_+(x,v)}\sigma(x+ sv,v)ds}\phi(x+\tau_+(x,v)v,v),\label{J1b}
\end{equation}
for $(x,v)\in X\times V$, and
\begin{equation}
\tilde j(1-\rho)\tilde J\phi(x,v)=e^{-\int_0^{\tau_+(x,v)}\sigma(x+ sv,v)ds}\phi(x+\tau_+(x,v)v,v),\label{J2}
\end{equation}
for $(x,v)\in \Gamma_-$ ($\rho(y)=0$ for $y\in \pa X$).

From \eqref{J2} and Lemma \ref{ref_lip} ($\tau_+$ is a Lipschitz function on $\Gamma_-$) it follows that $\tilde j(1-\rho)\tilde J\phi\in{\rm Lip}(\Gamma_-)$ and that there exists a constant $C(X,\sigma)$ which depends on $X$, $\|\sigma\|_\infty$ and ${\rm Lip}(\sigma)$ such that
\begin{equation*}
{\rm Lip}(\tilde j(1-\rho)\tilde J\phi)\le C(X,\sigma)({\rm Lip}(\phi)+\|\phi\|_\infty).
\end{equation*} 
This inequality proves that $\tilde j(1-\rho)\tilde J$ is bounded from ${\rm Lip}(\Gamma_+)$ to ${\rm Lip}(\Gamma_-)$.

From \eqref{J1a} and Lemma \ref{ref_lip} ($\tau_+\in C^1({\rm supp}\rho\times V)$), it follows that $\rho\tilde J\phi\in{\rm Lip}(X\times V)$ and that there exists a constant $C(X,\sigma)$ which depends on $X$, $\|\sigma\|_\infty$ and ${\rm Lip}(\sigma)$ so that
\begin{equation*}
{\rm Lip}(\rho\tilde J\phi)\le C(X,\sigma)({\rm Lip}(\phi)+\|\phi\|_\infty).
\end{equation*} 
This inequality proves that $\rho\tilde J$ is bounded from ${\rm Lip}(\Gamma_+)$ to ${\rm Lip}(X\times V)$.

From \eqref{J1b} it follows that ${\rm supp}((1-\rho)\tilde J\phi)\cap({\rm supp}_Xk\times V)=\emptyset$. Hence $\tilde K(1-\rho)\tilde J\phi=0$.
\end{proof}

\begin{corollary}
\label{cor_back}
Under conditions \eqref{INV}, \eqref{HL} and \eqref{HL1},
the backward  albedo operator $A_{\rm back}$ is a bounded operator from ${\rm Lip}(\Gamma_+)$ to ${\rm Lip}(\Gamma_-)$.
\end{corollary}

\begin{proof}[Proof of Corollary \ref{cor_back}]
Let $\phi\in{\rm Lip}(\Gamma_+)$. We write 
\begin{equation}
A_{\rm back}\phi=\tilde j(I-\tilde K)^{-1}\rho\tilde J\phi+\tilde j(I-\tilde K)^{-1}(1-\rho)\tilde J\phi.
\end{equation}
Since $\tilde K(1-\rho)\tilde J\phi=0$, then $(I-\tilde K)^{-1}(1-\rho)\tilde J\phi=(1-\rho)\tilde J$.
Hence we obtain
\begin{equation}
A_{\rm back}=\tilde j(I-\tilde K)^{-1}\rho\tilde J+\tilde j(1-\rho)\tilde J.
\end{equation}
Hence by Corollary \ref{cor_inv} and Lemma \ref{lem_ext} we obtain that $A_{\rm back}$ is a bounded operator from  from ${\rm Lip}(\Gamma_+)$ to ${\rm Lip}(\Gamma_-)$.
\end{proof}

Corollary \ref{cor_back} is used for the proof of Theorem \ref{thm:appsource}. It is established under the assumption that $k$ vanishes on the boundary of the domain (see \eqref{HL}). Then (under the additional assumption \eqref{INV}) solutions of the backward transport equation with Lipschitz boundary condition at  $\Gamma_+$ remain Lipschitz inside $X\times V$ and one can take their traces on $\Gamma_-$. It also allows us to easily study $\tilde K$ in the Banach space ${\rm Lip (X\times V)}$. Whether this assumption can be weakened to allow for scattering coefficients that do not vanish on the boundary is left open. 

\section{Estimates on multiple scattering kernels}
\label{sec:estms}

This central section presents the main technical results on the support of multiple scattering contributions in the albedo operator. The detailed proofs of the results,which revisit and generalize contributions in our earlier works, are postponed to section \ref{sec_app}.

Throughout the section, we assume that $X$ is a bounded convex domain of class $C^1$. 
Under the general assumption \eqref{Hyp},
let us denote by $\alpha_m$ the distributional kernel of the operator $j_+ K^m  J$ (a map from $L^1(\Gamma_-,d\xi)$ to $L^1(\Gamma_+,d\xi)$) and $\gamma_m$ the distributional kernel of the operator  $j_+K^m$ (a map from $L^1(X\times V)$ to $L^1(\Gamma_+,d\xi)$), $m\in \N$. This section provides explicit expressions for these distributional kernels (Lemma \ref{kernel}) as well as estimates when the scattering coefficient $k$ is bounded (Lemma \ref{est2}). We next apply these estimates to control multiple scattering contributions to the albedo operator 
(Propositions \ref{ballistic_test} and \ref{singlescat_test}).

By induction on $m\ge 2$, we define the measurable function $E$ on $X^m$ by the formula
\begin{equation}
  \label{eq:Ek}
  E(x_1,\ldots,x_m) = E(x_1,\ldots,x_{m-1})E(x_{m-1},x_m)\ 
\end{equation}
for a.e. $(x_1,\ldots,x_m)\in X^m$, where the function $E$ is initially defined as a measurable function on $ X^2$ by formula \eqref{eq:Exy}.
We use the same notation for the measurable function $E$ defined on $\pa X\times X^{m-2}\times \pa X$ or on $\pa X\times X^{m-1}$.

For $w\in \R^n$, $w\not=0$, we introduce the notation $\hat w:={w\over |w|}$.

\begin{lemma}
\label{kernel}
Under condition \eqref{Hyp}, the distributional kernel of $j_+K^m$ is the measurable function $\gamma_m$ given on $\Gamma_+\times X\times V$ by the formula
\begin{eqnarray}
\hspace{-2cm}
&&\gamma_m(z_0,v_0,z_m,v_m)=\int_0^{\tau_-(z_0,v_0)}\int_{X^{m-2}}\left[{E(z_0,\ldots, z_m)\over \Pi_{i=1}^{m-1}|z_i-z_{i+1}|^{n-1}}\right.\label{t6}\\
\hspace{-2cm}
&&\left.\times \Pi_{i=1}^mk(z_i,v_i,v_{i-1})\right]_{ v_i=\widehat{z_i-z_{i+1}},\ i=1\ldots m-1,\ z_1=z_0-tv_0}\nonumber
  dtdz_2\ldots dz_{m-1},\nonumber
\end{eqnarray}
for a.e. $(z_0,v_0,z_m,v_m)\in \Gamma_+\times X\times V$, $m\ge 3$. 
The distributional kernel of $j_+KJ$ is the distribution  $\alpha_1\in C_0(\Gamma_+\times\Gamma_-)'$ defined by the formula
\begin{eqnarray}
&&\alpha_1(z_0,v_0,z_2,v_1)=\nonumber\\
&&\int_0^{\tau_-(z_0,v_0)}\hskip -1cm E(z_0,z_0-tv_0,z_0-tv_0-\tau_-(z_0-tv_0,v_1)v_1)k(z_0-tv_0,v_1,v_0)\nonumber\\
&&\times\delta_{z_0-tv_0-\tau_-(z_0-tv_0,v_1)v_1}(z_2)dt,\ (z_0,v_0,z_2,v_1)\in \Gamma_+\times \Gamma_-,\label{t7b}
\end{eqnarray}
and for $m\ge 2$ the distributional kernel of $j_+K^mJ$ is the measurable function $\alpha_m$ on $\Gamma_+\times \Gamma_-$ given by
\begin{eqnarray}
&&\alpha_2(z_0,v_0,z_3,v_2)=|\nu(z_3)\cdot v_2|\int_0^{\tau_+(z_3,v_2)}\int_0^{\tau_-(z_0,v_0)}\!\!\!\!\!\!\!E(z_0,z_0-tv_0,z_2,z_3)\nonumber\\
&&\times\Big({k(z_0-tv_0,v_1,v_0)k(z_2,v_2,v_1)\over |z_0-tv_0-z_2|^{n-1}}
\Big)_{\big|{z_2=z_3+sv_2\atop 
v_1= \widehat{z_0-tv_0-z_2}}}dt ds\label{t7c}
\end{eqnarray}
for a.e. $(z_0,v_0,z_2,v_1)\in \Gamma_+\times \Gamma_-$, while
\begin{eqnarray}
&&\alpha_m(z_0,v_0,z_{m+1},v_m)=\nonumber\\
&&|\nu(z_{m+1})\cdot v_m|\int_0^{\tau_+(z_{m+1},v_m)}\int_0^{\tau_-(z_0,v_0)}\int_{X^{m-2}}\left[{E(z_0,\ldots, z_{m+1})\over \Pi_{i=1}^{m-1}|z_i-z_{i+1}|^{n-1}}\right.\label{t7}\\
&&\hskip -2cm\left.\times \Pi_{i=1}^mk(z_i,v_{i-1},v_i)\right]_{|z_1=z_0-tv_0,\ z_m=z_{m+1}+sv_m,\ v_i=\widehat{z_i-z_{i+1}},\ i=1\ldots m}\nonumber
  dt dsdz_2\ldots dz_{m-1},
\end{eqnarray}
for a.e. $(z_0,v_0,z_{m+1},v_m)\in \Gamma_+\times \Gamma_-$, $m\ge 3$.

The distributional kernel of $K^m$ is the measurable function $\beta_m$ also given by the right-hand side of \eqref{t6} for a.e.  $(z_0,v_0,z_m,v_m)\in (X\times V)^2$, $m\ge 3$.
\end{lemma}
The proof of Lemma \ref{kernel} is given in Section \ref{sec_app}.

In the sequel, we will use the following well known results:
\begin{lemma}
\label{lem_int}
For any $(z_0,z)\in \bar X^2$, $z_0\not=z$, we have
\begin{equation}
\int_{\R^n}{dz_1\over |z_0-z_1|^m|z_1-z|^{n-1}}\le {C_m\over |z_0-z|^{m-1}},\ 2\le m\le n-1,\textrm{ when }n\ge 3,\label{f2}
\end{equation}
while
\begin{eqnarray}
&&\int_X{dz_1\over |z_0-z_1||z_1-z|^{n-1}}\le C-C'\ln(|z_0-z|),\label{f1}\\
&&\Big|\int_X{\ln(|z_0-z_1|)dz_1\over|z_1-z|^{n-1}} \Big|\le C_n,\label{f3}
\end{eqnarray}
for some constants $C$, $C'$, $C_m$, $2\le m\le n$, where the constants $C$ and $C_n$ depend only on the diameter of the bounded domain $X$.
\end{lemma}

In the next Lemma, we denote $\sigma_-:=\min(0,\sigma)$.
\begin{lemma}
\label{est2}
Assume that $k\in L^\infty(X\times V^2)$. Then we have
\begin{equation}
\big|{\gamma_{n+2}(z_0,v_0,z,v)\over \tau_-(z_0,v_0)}\big|\le C_{n+2}e^{(n+2)\|\tau\sigma_-\|_\infty}\|k\|_\infty^{n+2},\label{e1}
\end{equation}
for a.e. $(z_0,v_0,x,v)\in \Gamma_+\times X\times V$, where the constant $C_{n+2}$ depends on the diameter of $X$,
\begin{equation}
\begin{array}{rcl}\displaystyle
&& \displaystyle {|\alpha_m(z_0,v_0,z_{m+1},v_m)|\over|\nu(z_{m+1})\cdot v_m|}
\\&\le& \displaystyle  C_m e^{(m+1)\|\tau\sigma_-\|_\infty}\|k\|_\infty^m\int_0^{\tau_-(z_0,v_0)}\hskip -3mm\int_0^{\tau_+(z_{m+1},v_m)}\hskip-1cm{dt ds\over |z_0-tv_0-z_{m+1}-sv_m|^{n-m+1}},\label{e4}
\end{array}\end{equation}
for a.e. $(z_0,v_0,z_{m+1},v_m)\in \Gamma_+\times\Gamma_-$  and some constant $C_m$ for $2\le m\le n$, and
\begin{equation}
\Big\|{\alpha_{n+1}(z_0,v_0,z_{n+2},v_{n+1})\over \tau_-(z_0,v_0)|\nu(z_{n+2})\cdot v_{n+1}|}\Big\|_{L^\infty({\Gamma_+}_{z_0,v_0}\times{\Gamma_-}_{z_{n+2},v_{n+2}})}\le C_{n+1}e^{(n+2)\|\tau\sigma_-\|_\infty}\|k\|_\infty^{n+1},\label{e2}
\end{equation}
for any $n\ge 2$, where the constant $C_{n+1}$ depends on the diameter of $X$.
\end{lemma}
The proof of Lemma \ref{est2} is given in Section \ref{sec_app}.

\begin{proposition}
\label{ballistic_test}
Assume that \eqref{INV} holds and that $k\in L^\infty(X\times V^2)$.
Let $(z',v')\in \Gamma_-$ and let $\eta>0$, and let $\phi\in L^\infty(\Gamma_+)$ so that ${\rm supp}\phi\subset\{(z,v)\in \Gamma_+\ |\ |z-z'-\tau_+(z',v')v'|\le \eta,\ |v-v'|\le \eta\}$ and $\|\phi\|_\infty\le 1$. Then there exist positive constants such that
\begin{equation}
\sup_{f\in L^1(\Gamma_-,d\xi)\atop \|f\|_{L^1(\Gamma_-,d\xi)}=1}\big|\int_{\Gamma_+}\phi(z,v)j_+KJf(z,v) d\xi(z,v)\big|\le C_1e^{2\|\tau\sigma_-\|_\infty}\|k\|_{\infty}\eta^{n-1},\label{i6}
\end{equation}
\begin{equation}
{\big|\int_{\Gamma_+}\alpha_n(z,v,z'',v'')\phi(z,v)d\xi(z,v)\big|\over |\nu(z'')\cdot v''|}\le C_n e^{(n+1)\|\tau\sigma_-\|_\infty}\|k\|_{\infty}^n\eta^{2n-2}(|\ln(\eta)|+1),\label{i1}
\end{equation}
\begin{equation}
{\big|\int_{\Gamma_+}\alpha_{n+1}(z,v,z'',v'')\phi(z,v)d\xi(z,v)\big|\over |\nu(z'')\cdot v''|}\le C_{n+1}e^{(n+2)\|\tau\sigma_-\|_\infty}\|k\|_{\infty}^{n+1}\eta^{2n-2},\label{i4}
\end{equation}
\begin{equation}
\sup_{f\in L^1(\Gamma_-,d\xi)\atop\|f\|_{L^1(\Gamma_-,d\xi)}=1}\Big|\int_{\Gamma_+}\phi(x,v)j_+K^{n+2}(I-K)^{-1}Jf(x,v)d\xi(x,v)\Big|\le C_{n+2}e^{(n+2)\|\tau\sigma_-\|_\infty}
\|k\|_{\infty}^{n+2}\eta^{2n-2},\label{i5}
\end{equation}
for any $n\ge 2$ and for a.e. $(z'',v'')\in \Gamma_-$,
and 
\begin{equation}
{\big|\int_{\Gamma_+}\alpha_m(z,v,z'',v'')\phi(z,v)d\xi(z,v)\big|\over |\nu(z'')\cdot v''|}\le C_m e^{(m+1)\|\tau\sigma_-\|_\infty}\|k\|_{\infty}^m\eta^{n+m-2},\textrm{ for }2\le m\le n-1,\label{i2}
\end{equation}
when $n\ge 3$ and for a.e. $(z'',v'')\in \Gamma_-$. 
The constants $C_m$, $1\le m\le n+1$, depend on the diameter of $X$ while $C_{n+2}$ depends on the diameter of $X$ and subcritical conditions.
\end{proposition}
The proof of Proposition \ref{ballistic_test} is given in Section \ref{sec_app}.
\begin{proposition}
\label{singlescat_test}
Let $n\ge 3$. Assume that \eqref{INV} holds and that $k\in L^\infty(X\times V^2)$.
Let $(z',v')\in \Gamma_-$ and let $\eta>0$, and let $\phi\in L^\infty(\Gamma_+)$ so that ${\rm supp}\phi\subset\{(z,v)\in \Gamma_+\ |\ d(z,v)\le \eta\}$ and $\|\phi\|_\infty\le 1$, where $d(z,v)$ denotes the distance between the lines $z'+\R v'$ and $z+\R v$. Then
there exist positive constants such that
\begin{equation}
{\big|\int_{\Gamma_+}\alpha_{n-1}(z,v,z'',v'')\phi(z,v)d\xi(z,v)\big|\over |\nu(z'')\cdot v''|}\le C_{n-1}e^{n\|\tau\sigma_-\|_\infty}\|k\|_{\infty}^{n-1} \eta^{n-2}(1+|\ln(\eta)|),\label{j1b}
\end{equation}
\begin{eqnarray}
{\big|\int_{\Gamma_+}\alpha_n(z,v,z'',v'')\phi(z,v)d\xi(z,v)\big|\over |\nu(z'')\cdot v''|}\le C_n e^{(n+1)\|\tau\sigma_-\|_\infty}\|k\|_{\infty}^n\eta^{n-2},&&\label{j2}\\
{\big|\int_{\Gamma_+}\alpha_{n+1}(z,v,z'',v'')\phi(z,v)d\xi(z,v)\big|\over |\nu(z'')\cdot v''|}\le C_{n+1}e^{(n+2)\|\tau\sigma_-\|_\infty} \|k\|_{\infty}^{n+1} \eta^{n-2},&&\label{j3}
\end{eqnarray}
\begin{equation}
\sup_{f\in L^1(\Gamma_-,d\xi)\atop\|f\|_{L^1(\Gamma_-,d\xi)}=1}\Big|\int_{\Gamma_+}\phi(x,v)j_+K^{n+2}(I-K)^{-1}Jf(x,v) d\xi(x,v)\Big|\le C_{n+2}e^{(n+2)\|\tau\sigma_-\|_\infty}\|k\|_{\infty}^{n+2}\eta^{n-2},\label{j4}
\end{equation}
for any $n\ge 3$ and for a.e. $(z'',v'')\in \Gamma_-$, and
\begin{equation}
{\big|\int_{\Gamma_+}\alpha_m(z,v,z'',v'')\phi(z,v)d\xi(z,v)\big|\over |\nu(z'')\cdot v''|}\le C_m e^{(m+1)\|\tau\sigma_-\|_\infty}\|k\|_{\infty}^m\eta^{m-1},\textrm{ for }2\le m\le n-2,\label{j1}
\end{equation}
when $n\ge 4$  and for a.e. $(z'',v'')\in \Gamma_-$. The constants $C_m$, $2\le m\le n+1$, depend on the diameter of $X$ while $C_{n+2}$ depends on  the diameter of $X$ and subcritical conditions.
\end{proposition}
The proof of Proposition \ref{singlescat_test} is also given in Section \ref{sec_app}.

\section{Proof of Theorems \ref{thm:exactsource}, \ref{thm:appsource} and \ref{thm:stabcoefs}}
\label{sec:inv}

Using the estimates presented in the preceding section, we are now ready to prove the main results of the paper. We first recall the following Lemma.

\begin{lemma}[see \cite{McDST-IP-10}]
\label{lem_unity}
There is a family of maps $\psi_{\rho,x_0',\theta_0'}\in L^1(\Gamma_-,d\xi)$, $(x_0',\theta_0')\in \Gamma_-$ and $\rho>0$, such that $\|\psi_{\rho,x_0',\theta_0'}\|_{L^1(\Gamma_-,\xi)}=1$ and for any $f\in L^\infty(\Gamma_-,d\xi)$ given,
\begin{equation}
\lim_{\rho\to 0^+}\int_{\Gamma_-}\psi_{\rho,x_0',\theta_0'}(x',\theta')f(x',\theta')d\xi(x',\theta')=f(x'_0,\theta_0'),\label{j20}
\end{equation}
whenever $(x_0',\theta_0')$ is in the Lebesgue set of $f$. In particular, \eqref{j20} holds for almost every $(x_0',\theta_0')\in \Gamma_-$.
\end{lemma}

\subsection{Proof of Theorem \ref{thm:exactsource}}
We assume that assumptions \eqref{INV} and \eqref{HL} are satisfied and that $X$ is of class $C^1$. First let $\phi\in  {\rm Lip}(\Gamma_+)$, ${\rm Lip}\phi\le \kappa$. From \eqref{eq:decalb} and \eqref{eq:depart} (with $g=g_j=\psi_{\rho,x'',v''}$, $j=1,2$)
it follows that
\begin{equation}
|\l(B_1-B_2)\psi_{\rho,x'',v''},\phi\r+\l(S_1-S_2)\psi_{\rho,x'',v''},\phi\r+\l(M_1-M_2)\psi_{\rho,x'',v''},\phi\r|\le\ep,\label{j21}
\end{equation}
for $(x'',v'')\in \Gamma_-$ and $\rho>0$.
Then by Lemma \ref{lem_unity}
\begin{eqnarray}
&&\l(B_1-B_2)\psi_{\rho,x'',v''},\phi\r\nonumber\\
&=&\int_{\Gamma_+}\phi(x,v)(E_1-E_2)(x-\tau_-(x,v)v,x)\psi_{\rho,x'',v''}(x-\tau_-(x,v)v,v)d\xi(x,v)\nonumber\\
&=&\int_{\Gamma_-}\phi(x+\tau_+(x,v)v,v)(E_1-E_2)(x,x+\tau_+(x,v)v)\psi_{\rho,x'',v''}(x,v)d\xi(x,v)\nonumber\\
&\to&\phi(x''+\tau_+(x'',v'')v'',v'')(E_1-E_2)(x'',x''+\tau_+(x'',v'')v'') \label{j22}
\end{eqnarray}
as $\rho\to 0^+$, for $(x'',v'')\in \Gamma_-$ (the function $\phi(x+\tau_+(x,v)v,v)(E_1-E_2)(x,x+\tau_+(x,v)v)$ is continuous in $(x,v)\in \Gamma_-$).
Similarly from \eqref{t7b} it follows that
\begin{eqnarray}
&&\l(S_1-S_2)\psi_{\rho,x'',v''},\phi\r
=\int_{\Gamma_-}\int_V\int_0^{\tau_+(x,v)}(\tilde E_1k_1-\tilde E_2k_2)(x+sv,v,v')\nonumber\\
&&\times\phi(x+sv+\tau_+(x+sv,v')v',v')ds dv'\psi_{\rho,x'',v''}(x,v) d\xi(x,v)\nonumber\\
&\to&\int_V\int_0^{\tau_+(x'',v'')}(\tilde E_1k_1-\tilde E_2k_2)(x''+sv'',v'',v')\nonumber\\
&&\times\phi(x''+sv''+\tau_+(x''+sv'',v')v',v')ds dv',\textrm{ as }\rho\to 0^+,\label{j23}
\end{eqnarray}
for $(x'',v'')\in \Gamma_-$  (the function $\int_V\int_0^{\tau_+(x,v)}(\tilde E_1k_1-\tilde E_2k_2)(x+sv,v,v')
\phi(x+sv+\tau_+(x+sv,v')v',v')ds dv'$ is continuous in $(x,v)\in \Gamma_-$).
Finally
\begin{eqnarray}
&&\l(M_1-M_2)\psi_{\rho,x'',v''},\phi\r-\int_{\Gamma_+}\phi(x,v)j_+K^{n+2}(I-K)^{-1}J\psi_{\rho,x'',v''}(x,v)d\xi(x,v)\nonumber\\
&=&\sum_{m=2}^{n+1}\int_{\Gamma_+\times\Gamma_-}(\alpha_{m,1}-\alpha_{m,2})(x,v,x',v')\phi(x,v)\psi_{\rho,x'',v''}(x',v')d\xi(x',v') d \xi(x,v)\nonumber\\
&\to&\sum_{m=2}^{n+1}\int_{\Gamma_+}(\alpha_{m,1}-\alpha_{m,2})(x,v,x'',v'')\phi(x,v)d \xi(x,v),\label{j24}
\end{eqnarray}
as $\rho\to 0^+$ when $(x'',v'')$ belongs to the intersection ${\mathcal L}_\phi$ of the Lebesgue sets ${\mathcal L}_{j,\phi}$ of the functions $\int_{\Gamma_+}(\alpha_{j,1}-\alpha_{j_2})(x,v,x'',v'')\phi(x,y)d \xi(x,y)$ that belong to $L^\infty(\Gamma_-)$ for $j=2\ldots n+1$.

We now chose our function $\phi$ to derive the first estimate of the Theorem. Let $(x_0,v_0)\in \Gamma_-$,  $\eta\in (0,\diam)$ and let $\phi\in {\rm Lip}(\Gamma_+)$, ${\rm Lip \phi}\le \kappa$ so that 
\begin{equation}
{\rm supp}\phi\subset\{(x,v)\in \Gamma_+\ |\ |x-x_0-\tau_+(x_0,v_0)v_0|\le \eta,\ |v-v_0|\le \eta\},\label{p1a}
\end{equation}
\begin{eqnarray}
&&\phi(x_0+\tau_+(x_0,v_0)v_0,v_0)=\|\phi\|_{L^\infty(\Gamma_+)}\le 1,\label{p1b}\\
&&C_X^{-1}\kappa \eta\le \|\phi\|_{\infty}\le C_X\kappa \eta,\label{p1}
\end{eqnarray}
for some constant $C_X$ that depends on $X$.

We set $y_0=x_0+\tau_+(x_0,v_0)v_0$.
By \eqref{p1a} and Proposition \ref{ballistic_test} we have 
\begin{eqnarray}
&&\big|\int_V\int_0^{\tau_+(x'',v'')}
\phi(x''+sv''+\tau_+(x''+sv'',v')v',v')\nonumber\\
&&\times (\tilde E_1k_1-\tilde E_2k_2)(x''+sv'',v'',v')ds dv'\big|\nonumber\\
&&\le C\max(\|k_1\|_\infty, \|k_2\|_\infty)\|\phi\|_\infty\eta^{n-1},\label{j41}
\end{eqnarray}
\begin{equation}
\sup_{\rho}\Big|\int_{\Gamma_+}\phi(x,v)j_+K^{n+2}(I-K)^{-1}J\psi_{\rho,x'',v''}(x,v)d\xi(x,v)\Big|
\le C\|\phi\|_\infty\eta^{2n-2},\label{j42}
\end{equation}
for $(x'',v'')\in \Gamma_-$, and
\begin{equation}
\Big|\sum_{m=2}^{n+1}\int_{\Gamma_+}(\alpha_{m,1}-\alpha_{m,2})(x,v,x'',v'')\phi(x,v)d \xi(x,v)\Big|
\le C\|\phi\|_\infty \eta^{n-1},\label{j43}
\end{equation}
for a.e. $(z'',v'')\in \Gamma_-$ (more precisely, we verify that \eqref{j41} follows from \eqref{i6} and \eqref{j23}). Combining \eqref{j21}--\eqref{j24} and \eqref{j41}--\eqref{j43}, we thus obtain
\begin{equation}
|\phi(x''+\tau_+(x'',v'')v'',v'')||E_1-E_2|(x'',x''+\tau_+(x'',v'')v'')\le \ep +C\|\phi\|_{L^\infty(\Gamma_+)}\eta^{n-1}, \label{j43a}
\end{equation}
for $(x'',v'')\in \Gamma_-$, where $C$ depends on $X$, the subcritical condition, and on $\|k\|_\infty$ (the left-hand side of \eqref{j43a} is continuous in $(x'',v'')$). 
Set $(x'',v'')=(x_0,v_0)$ in the latter estimate and use \eqref{p1} to obtain
\begin{eqnarray}
|(E_1-E_2)(x_0,y_0)|\le C({\ep\over \kappa \eta}+\eta^{n-1}).\label{j43b}
\end{eqnarray}

Note that from  \eqref{p1} and  from the estimate $\|\phi\|_{L^\infty}\le 1$ it follows that 
$\kappa\eta\le C_X$. In addition we impose $\eta\le C_X^{-1}\kappa^{-1}$ so that the constraint $\|\phi\|_{L^\infty}\le 1$ is a consequence of \eqref{p1}. Set $\tilde C_X=\min(C_X,C_X^{-1})$, $\kappa'={\kappa\diam\over \tilde C_X}$ and $\eta'={\eta\over \diam}$ and we obtain that 
\begin{equation}
|(E_1-E_2)(x_0,y_0)|\le C'({\ep\over \kappa' \eta'}+\eta'^{n-1}),\ \eta'\in (0,1),\ \kappa'\eta'\in (0,1),\label{j43d}
\end{equation}
for some constant $C'$ that  depends on $X$, the subcritical condition, and on $\|k\|_\infty$.
Then we derive \eqref{eq:stab1} from \eqref{j43d} and 
\begin{equation}
\min_{\eta'\in(0,1),\ \kappa'\eta'\le 1}({\ep\over \kappa' \eta'}+\eta'^{m-1})\le c_m\Big(\Big(\dfrac \eps {\kappa'} \Big)^{\frac{m-1}m}\vee \eps\Big),\label{j43e}
\end{equation}
for a constant $c_m$ that only depends on the integer $m\ge 2$.

We now prove \eqref{eq:stab2a} and \eqref{eq:stab2b}. We introduce the following piecewise affine functions $\chi_1\in C([0,+\infty),\R)$  and $\chi_2\in C(\R,\R)$
\begin{eqnarray*}
\chi_1(t)=
\left\lbrace
\begin{array}{l}
1\textrm{ when }t\le 1\\
(2-t)\textrm{ when }1\le t\le 2\\
0\textrm{ when }t\ge 2
\end{array}
\right.,\ 
\chi_2(t)=
\left\lbrace
\begin{array}{l}
1\textrm{ when }t\ge 1\\
t\textrm{ when }-1\le t\le 1\\
-1\textrm{ when }t\le -1
\end{array}
.\right.
\end{eqnarray*}

Let $(x_0,v_0)\in \Gamma_-$ and let $\eta\in (0,\min(\diam,\kappa^{-1}))$. 
We extend the functions $k_1$ and $k_2$ by  $0$ to $\R^n\times V\times V$ and extend the functions $\sigma_1$ and $\sigma_2$ by  $0$ to $\R^n\times V$. Under assumption \eqref{HL}, the function $G$, given by 
\begin{equation}
G(x,v):=
\left\lbrace
\begin{array}{l}
(1-(v_0\cdot v)^2)(\tilde E_1k_1-\tilde E_2 k_2)(x-{(x-x_0)\cdot (v-(v\cdot v_0) v_0)\over 1-(v\cdot v_0)^2} v,v_0,v)\textrm{ when }v\not=\pm v_0,\\
0 \textrm{ otherwise },
\end{array}
\right.\label{j50}
\end{equation}
for $(x,v)\in\Gamma_+$, defines a Lipschitz function on $\Gamma_+$ (for such a statement we actually use that ${(x-x_0)\cdot (v-(v\cdot v_0) v_0)\over 1-(v\cdot v_0)^2}$ is a bounded function of $(x,v)$ on the support of the function $G$). 

We introduce the function
$\tilde {\rm d}(x,v):=\min_{(t,s)\in (-\diam,\diam)^2}|x-x_0-tv+sv_0|$. The function $\tilde {\rm d}$ is $\diam$-Lipschitz function on $\Gamma_+$ and it satisfies
$$
{\rm d}(x,v):=\min_{(t,s)\in \R^2}|x-x_0-tv+sv_0|\le \tilde {\rm d}(x,v).
$$
We set 
\begin{equation}
\phi(x,v):=\kappa\eta\chi_2\Big({G(x,v)\over \eta({\rm Lip G}+1)}\Big)\chi_1\big({\tilde {\rm d}(x,v)\over \diam \eta}\big) (1-\chi_1)\big({|v-v_0|\over\eta}\big)\textrm{ for }(x,v)\in \Gamma_+.\label{j44a}
\end{equation}
The product in $\chi_2$ on the right hand side of \eqref{j44a} loosely speaking introduces an approximation of the sign of the function $\tilde E_1k_1-\tilde E_2 k_2$. 

Then $\phi\in{\rm Lip}(\Gamma_+)$ and 
\begin{eqnarray}
&&{\rm Lip}(\phi)\le\kappa, {\rm supp \phi}\subseteq\{(x,v)\in \Gamma_+\ |\ {\rm d}(x,v)\le 2\diam \eta \},\label{j44}\\
&& \|\phi\|_{L^\infty(\Gamma_+)}\le \kappa\eta\le 1,\ \phi(y_0,v_0)=0.\label{j44b}
\end{eqnarray}

By Proposition \ref{singlescat_test}, we have
\begin{equation}
\Big|\sum_{m=2}^{n+1}\int_{\Gamma_+}(\alpha_{m,1}-\alpha_{m,2})(x,v,x'',v'')\phi(x,v)d \xi(x,v)\Big|\le 
\left\lbrace
\begin{array}{l}
 C\|\phi\|_\infty \eta(1+|\ln(\eta)|),\textrm{ }n=3,\\
 C\|\phi\|_\infty \eta\textrm{ when }n\ge 4,
\end{array}
\right.\label{j45}
\end{equation}
\begin{equation}
\sup_{\rho}\Big|\int_{\Gamma_+}\phi(x,v)j_+K^{n+2}(I-K)^{-1}J\psi_{\rho,x'',v''}(x,v)d\xi(x,v)\Big|
\le C\|\phi\|_\infty\eta^{n-2},\label{j46}
\end{equation}
for a.e. $(x'',v'')\in \Gamma_-$.
Since $\phi(y_0,v_0)=0$, the contribution of the ballistic term vanishes and from \eqref{j21}--\eqref{j24} and \eqref{j45} and \eqref{j46}, we deduce that
\begin{eqnarray}
&&\int_V\int_0^{\tau_+(x_0,v_0)}\phi(x_0+sv_0+\tau_+(x_0+sv_0,v)v,v)(\tilde E_1k_1-\tilde E_2 k_2)(x_0+sv_0,v_0,v)ds dv\nonumber\\
&&\le\ep+\left\lbrace
\begin{array}{l}
 C\|\phi\|_\infty \eta(1+|\ln(\eta)|)\textrm{ when }n=3,\\
 C\|\phi\|_\infty \eta\textrm{ when }n\ge 4.
\end{array}
\right.\label{j47}
\end{eqnarray}
From \eqref{j44a} and the definition of the functions $\chi_1$ and $\chi_2$, we estimate the contribution of the single scattering for $(x'',v'')=(x_0,v_0)$ as
\begin{eqnarray}
&&\int_V\int_0^{\tau_+(x_0,v_0)}\phi(x_0+sv_0+\tau_+(x_0+sv_0,v)v,v)(\tilde E_1k_1-\tilde E_2 k_2)(x_0+sv_0,v_0,v)ds dv\nonumber\\
&=&\kappa\eta\int_V\int_0^{\tau_+(x_0,y_0)}\chi_2\Big({G(x_0+sv_0+\tau_+(x_0+sv_0,v)v,v)\over \eta({\rm Lip G}+1)}\Big)\nonumber\\
&&\times(1-\chi_1)\big({|v-v_0|\over\eta}\big)
(\tilde E_1k_1-\tilde E_2 k_2)(x_0+sv_0,v_0,v)ds dv\nonumber\\
&\ge&\kappa \eta\int_V\int_0^{\tau_+(x_0,y_0)}\big|\tilde E_1k_1-\tilde E_2 k_2\big|(x_0+sv_0,v_0,v)(1-\chi_1)\big({|v-v_0|\over\eta}\big)ds dv\nonumber\\
&&-\kappa \eta\int_W
|\tilde E_1k_1-\tilde E_2 k_2|(x_0+sv_0,v_0,v)(1-\chi_1)\big({|v-v_0|\over\eta}\big)ds dv\nonumber
\end{eqnarray}
\begin{eqnarray}
&\ge&\kappa \eta\int_V\int_0^{\tau_+(x_0,y_0)}\big|\tilde E_1k_1-\tilde E_2 k_2\big|(x_0+sv_0,v_0,v)ds dv\nonumber\\
&&-\kappa \eta\int_W
|\tilde E_1k_1-\tilde E_2 k_2|(x_0+sv_0,v_0,v)ds dv\nonumber\\
&&-2\diam\kappa \eta(\|k_1\|_\infty+\|k_2\|_\infty)\int_{v\in V,\ |v-v_0|\le 2\eta} dv,\label{j48}
\end{eqnarray}
where 
\begin{equation}
W:= \{(s,v)\in (0,\tau_+(x_0,y_0))\times V\ |\ |G(x_0+sv_0+\tau_+(x_0+sv_0,v)v,v)|\le ({\rm Lip} (G)+1)\eta\},
\end{equation}
(we used the definition of $\chi_1$ and $\chi_2$ in the above estimates).
Note that by definition of the function $G$ \eqref{j50} and the function $\phi$ \eqref{j44a}, the integrand on the left-hand side of \eqref{j48} is nonnegative. Considering now the following inclusion (see definition of the function $G$)
\begin{eqnarray*}
W&\subseteq& W_1\cup W_2,\\
W_1&:=&\{(s,v)\in (0,\tau_+(x_0,y_0))\times V\ |\ \min(|v-v_0|,|v+v_0|)\le \eta\}\\
W_2&:=&\{(s,v)\in (0,\tau_+(x_0,y_0))\times V\ |\ \min(|v-v_0|,|v+v_0|)\ge \eta,\\
&& |\tilde E_1k_1-\tilde E_2 k_2|(x_0+sv_0,v_0,v)|\le {({\rm Lip} (G)+1)\eta\over 1-(v\cdot v_0)^2}\},
\end{eqnarray*}
we estimate the second term of the right hand side of \eqref{j48} as
\begin{eqnarray}
&&\hskip -1cm\int_W|\tilde E_1k_1-\tilde E_2 k_2|(x_0+sv_0,v_0,v)ds dv
\le C(\|k_1\|_\infty+\|k_2\|_\infty)\eta^{n-1}+\int_{W_2} {{({\rm Lip} (G)+1)\eta\over 1-(v\cdot v_0)^2}}ds dv\nonumber\\
&\le & C(\|k_1\|_\infty+\|k_2\|_\infty)\eta^{n-1}+({\rm Lip} (G)+1)\diam \eta\int_{v\in V,\ \min(|v-v_0|,|v+v_0|) \ge \eta} \hskip -3cm(1-(v\cdot v_0)^2)^{-1}dv.\label{j49a}
\end{eqnarray}
We next deduce the basic estimate for $\eta'\in (0,\diam)$
\begin{equation}
\sup_{v''\in V}\int_{v\in V,\ \min(|v-v''|,|v+v''|) \ge \eta'} {dv \over 1-(v\cdot v'')^2}\le
\left\lbrace
\begin{array}{l}
C(1+|\ln(\eta')|)\textrm{ when }n=3,\\
C\textrm{ when }n\ge 4.
\end{array}
\right. \label{j49b}
\end{equation}
Collecting \eqref{j48}--\eqref{j49b}, we obtain that  for some constant $C$ that depends on $\diam$, $\|k_j\|_\infty$, ${\rm Lip}(\sigma_j)$, ${\rm Lip}(k_j)$, $j=1,2$, we have
\begin{eqnarray*}
&&\hskip -1cm\int_V\int_0^{\tau_+(x_0,v_0)}\phi(x_0+sv_0+\tau_+(x_0+sv_0,v)v,v)(\tilde E_1k_1-\tilde E_2 k_2)(x_0+sv_0,v_0,v)ds dv\nonumber\\
&\ge&\kappa \eta\int_V\int_0^{\tau_+(x_0,y_0)}\big|\tilde E_1k_1-\tilde E_2 k_2\big|(x_0+sv_0,v_0,v)ds dv
-\kappa \eta^2\left\lbrace
\begin{array}{l}
C(1+|\ln(\eta)|),\textrm{ }n=3,\\
C\textrm{ when }n\ge 4.
\end{array}
\right.
\end{eqnarray*}
Hence, combining the latter estimate with \eqref{j47},we obtain that there exists a constant $C$ that depends on $X$, $\|k_j\|_\infty$, ${\rm Lip}(\sigma_j)$, ${\rm Lip}(k_j)$, $j=1,2$,
 and on the subcritical conditions \eqref{INV} such that
\begin{equation*}
\int_V\int_0^{\tau_+(x_0,y_0)}\big|\tilde E_1k_1-\tilde E_2 k_2\big|(x_0+sv_0,v_0,v)ds dv
\le{\ep\over \kappa\eta}+
C\left\lbrace
\begin{array}{l}
 \eta(1+|\ln(\eta)|),\textrm{}n=3,\\
 \eta\textrm{ when }n\ge 4.
\end{array}
\right.
\end{equation*}
Then in dimension $n\ge 4$ we derive \eqref{eq:stab2b} in the same way we derive \eqref{eq:stab1} from \eqref{j43b} (see \eqref{j43e} for $m=2$).
In dimension $n=3$, we set $\eta'=\diam^{-1}\eta$ and $\kappa'=\kappa\diam$ and we obtain 
\begin{equation}
\int_V\int_0^{\tau_+(x_0,y_0)}\big|\tilde E_1k_1-\tilde E_2 k_2\big|(x_0+sv_0,v_0,v)ds dv
\le C'\big({\ep\over \kappa'\eta'}+
 \eta'(1+|\ln(\eta')|)\big),\label{j50a}
\end{equation}
for a constant $C'$ that depends on $X$, $\|k_j\|_\infty$, ${\rm Lip}(\sigma_j)$, ${\rm Lip}(k_j)$, $j=1,2$,
 and on the subcritical conditions \eqref{INV}. Then we derive \eqref{eq:stab2a} from \eqref{j50} and the estimate
$$
\min_{\eta'\in (0,\min(1,\kappa'^{-1}))}\big({\ep\over \kappa'\eta'}+
 \eta'(1+|\ln(\eta')|)\big)\le c\Big(\Big(\dfrac \eps {\kappa'} \Big)^{\frac12}(1+\sqrt{|\ln({\ep\over{\kappa'}})|})\vee \eps\Big)
$$
for a universal constant $c$.

\subsection{Proof of Theorem \ref{thm:appsource}}
We assume that assumptions \eqref{INV}, \eqref{HL} and \eqref{HL1} are satisfied and start with the identity
\begin{equation}
\l(A_1-A_2)g,\phi\r=\ep_2-\ep_1+\l A_2(g_2-g)-A_1(g_1-g),\phi\r.
\end{equation}
Hence
\begin{eqnarray*}
|\l(A_1-A_2)g,\phi\r|&\le&\ep+|\l(g_2-g),A_{2,{\rm back}}\phi\r|+|\l(g_1-g),A_{1,{\rm back}}\phi\r|\\
&\le&\ep+W_{C\kappa}(g_2-g)+W_{C\kappa}(g_1-g)\le \ep+2C\delta,
\end{eqnarray*}
where $C$ is the maximum between the norms of $A_{1,{\rm back}}$ and $A_{2,{\rm back}}$ as bounded operators from ${\rm Lip}(\Gamma_+)$ to ${\rm Lip}(\Gamma_-)$.
Then the proof follows the same lines as the proof of Theorem \ref{thm:exactsource}.

\subsection{Proof of Theorem \ref{thm:stabcoefs}}
Let us denote $M=\max(\|\tau \sigma_1\|_\infty,\|\tau\sigma_2\|_\infty,\|\tau\sigma_{1,p}\|_\infty)$. Our main task is to select the right representative of $(\sigma_1,k_1)$. Set
\begin{equation}
\phi(x,v):=e^{\int_0^{\tau_-(x,v)}(\sigma_1- \sigma_2)(x-sv,v)ds-{\tau_-(x,v)\over\tau(x,v)}\int^{\tau_+(x,v)}_{-\tau_-(x,v)}(\sigma_1-\sigma_2)(x+sv,v)ds},
\end{equation}
for $(x,v)\in X\times V$. Then 
\begin{equation}
\phi>0,\ \ln(\phi)\in W^\infty,
\end{equation}
\begin{equation}
v\cdot \nabla_x \ln(\phi)(x,v)=(\sigma_1-\sigma_2)(x,v)-{\int_{-\tau_-(x,v)}^{\tau_+(x,v)}(\sigma_1-\sigma_2)(x+sv,v)ds\over \tau(x,v)},\ (x,v)\in X\times V,
\label{G1}
\end{equation}
\begin{equation}
\phi_{|\pa X\times V}=1.
\end{equation}
Then let us choose the following representative of $(\sigma_1,k_1)$:
\begin{equation}
\tilde \sigma_1:=\sigma_1-v\cdot\nabla_x \ln(\phi),\ \tilde k_1(x,v',v)={\phi(x,v)\over \phi(x,v')}k_1(x,v',v),
\end{equation}
for $(x,v',v)\in X\times V^2$. From \eqref{G1} it follows that
\begin{equation}
\tau(\tilde \sigma_1-\sigma_2)(x,v)=\int_{-\tau_-(x,v)}^{\tau_+(x,v)}(\sigma_1-\sigma_2)(x+sv,v)ds,
\end{equation}
for $(x,v)\in X\times V$. Hence using Theorem \ref{thm:exactsource} or \ref{thm:appsource} we have
\begin{eqnarray*}
C {\mathfrak e}_b(\eps,\delta)
&\ge&e^{\int_{-\tau_-(x,v)}^{\tau_+(x,v)}\sigma_2(x+sv,v)ds}|e^{\int_{-\tau_-(x,v)}^{\tau_+(x,v)}(\sigma_1-\sigma_2)(x+sv,v)ds}-1|\\
&\ge& e^{-3M}\big|\int_{-\tau_-(x,v)}^{\tau_+(x,v)}(\sigma_1-\sigma_2)(x+sv,v)ds\big|=e^{-3M}|\tau(\tilde \sigma_1-\sigma_2)|(x,v),
\end{eqnarray*}
for $(x,v)\in X\times V$. Hence we obtain the first estimate of the Theorem.

After straightforward computations, we have
\begin{eqnarray*}
&&\tilde k_1(x,v',v)-k_2(x,v',v)\\
&=&e^{\int_0^{\tau_-(x,v')}\sigma_2(x-sv',v')ds+\int_0^{\tau_+(x,v)}\sigma_2(x+sv,v)ds}(\tilde E_1k_1-\tilde E_2k_2)(x,v',v)\\
&&+k_1(x,v',v)e^{-\int_0^{\tau_-(x,v')}(\sigma_1- \sigma_2)(x-sv',v')ds-\int_0^{\tau_+(x,v)}(\sigma_1-\sigma_2)(x+sv,v)ds}\\
&&\times\big(e^{{\tau_-(x,v')\over \tau(x,v')}\int_{-\tau_-(x,v')}^{\tau_+(x,v')}(\sigma_1-\sigma_2)(x+sv',v')ds+{\tau_+(x,v)\over \tau(x,v)}
\int_{-\tau_-(x,v)}^{\tau_+(x,v)}(\sigma_1-\sigma_2)(x+sv,v)ds}-1\big),
\end{eqnarray*}
for $(x,v',v)\in X\times V^2$. Hence 
\begin{eqnarray*}
&&|\tilde k_1(x,v',v)-k_2(x,v',v)|
\le e^{2\|\tau\sigma_2\|_\infty}|\tilde E_1k_1-\tilde E_2k_2|(x,v',v)\\
&&+|k_1(x,v',v)|e^{4\|\tau(\sigma_1-\sigma_2)\|_\infty}\big(|\int_{-\tau_-(x,v')}^{\tau_+(x,v')}(\sigma_1-\sigma_2)(x+sv',v')ds|\nonumber\\
&&+|\int_{-\tau_-(x,v)}^{\tau_+(x,v)}(\sigma_1-\sigma_2)(x+sv,v)ds|\big)\\
&\le &e^{2M}|\tilde E_1k_1-\tilde E_2k_2|(x,v',v)+2e^{8M}|k_1(x,v',v)|{\mathfrak e}_b(\eps,\delta),
\end{eqnarray*}
for $(x,v',v)\in X\times V^2$ (we used the basic estimate $|e^s-1|\le e^R|s|$ for $s\in [-R,R]$, $R>0$).
Finally from Theorems \ref{thm:appsource} or \ref{thm:exactsource} (and $\|\tau\sigma_{1,p}\|_\infty\le M$) it follows that
\begin{eqnarray*}
&&\int_V\int_0^{\tau_+(x,v')}|\tilde k_1(x+sv',v',v)-k_2(x+sv',v',v)|ds dv\nonumber\\
&\le& Ce^{2M}{\mathfrak e}_s(\eps,\delta)
+2Ce^{8M}M{\mathfrak e}_b(\eps,\delta),
\end{eqnarray*}
for $(x,v')\in X\times V$ and some constant $C$.
The second estimate is proved.

\section{Proof of Lemmas \ref{kernel} and \ref{est2} and Propositions \ref{ballistic_test} and \ref{singlescat_test}}
\label{sec_app}


\begin{proof}[Proof of Lemma \ref{kernel}]
For $m\ge 2$, let $\beta_m$ denote the distributional kernel of the
operator $K^m$ where $K$ is defined by \eqref{eq:K}.  We first give
the explicit expression of $\beta_2$ and $\beta_3$. We then obtain the explicit expression of the kernel $\beta_m$ by induction and derive the explicit expression for $\gamma_m$.  


Let $\psi\in L^1(X\times V,\tau^{-1}dx dv)$. From \eqref{eq:K} it follows that
\begin{eqnarray}
&&\hspace{-1cm}K^2\psi(x,v)=\int_0^{\tau_-(x,v)}\int_V k(x-tv,v_1,v)\int_0^{\tau_-(x-tv,v_1)}\!\!\!\!\!\!\!\!\!\!\!\!\!E(x,x-tv,x-tv-t_1v_1)\nonumber\\
&&\hspace{-1cm}\times\int_V k(x-tv-t_1v_1,v',v_1)\psi(x-tv-t_1v_1,v')dv'dt_1dv_1dt,
\label{t1}
\end{eqnarray}
for a.e. $(x,v)\in  X\times V$.
Performing the change of variables $x'=x-tv-t_1v_1$ ($dx'=t_1^{n-1}dt_1dv_1$) on the right hand side of \eqref{t1}, we obtain
\begin{eqnarray}
K^2\psi(x,v)&=&\int_0^{\tau_-(x,v)}\int_X k(x-tv,\widehat{x-tv-x'},v)E(x,x-tv,x')\nonumber\\
&&\times\int_V{k(x',v',\widehat{x-tv-x'})\over  |x-tv-x'|^{n-1}}\psi(x',v')dv'dx'dt,
\label{t2b}
\end{eqnarray}
for a.e. $(x,v)\in X\times V$. Hence
\begin{equation}
\beta_2(x,v,x',v')=\int_0^{\tau_-(x,v)}\!\!\!\!\!\!\! E(x,x-tv,x'){k(x-tv,v_1,v)k(x',v',v_1)_{|v_1= \widehat{x-tv-x'}}\over |x-tv-x'|^{n-1}}dt,\label{t2}
\end{equation}
for a.e. $(x,v,x',v')\in (X\times V)^2$. From
\eqref{t2} and \eqref{eq:K}, it follows that
$$
K^3\psi(x,v)=\int_{(0,\tau_-(x,v))\times V}\!\!\!\!\!\!\!\!\!\!\!\!\!\!\!\!\!\!\!\!\!\!\!\!E(x,x-tv)k(x-tv,v_1,v)\int_{X\times V}\!\!\!\!\!\!\!\beta_2(x-tv,v_1,x',v')\psi(x',v')dx'dv'dv_1dt
$$
\begin{eqnarray}
&&\hspace{-2cm}=\int_{X\times V}\hspace{-0.5cm}\psi(x',v')\int_0^{\tau_-(x,v)}\!\!\!\int_V k(x-tv,v_1,v)
\int_0^{\tau_-(x-tv,v_1)}{E(x,x-tv,x-tv-t_1v_1,x')\over |x-tv-t_1v_1-x'|^{n-1}}\nonumber\\
&&\hspace{-2cm}\times k(x-tv-t_1v_1,v_2,v_1)k(x',v',v_2)_{v_2=\widehat{x-tv-t_1v_1-x'}}dt_1dv_1dtdx'dv',\label{t3}
\end{eqnarray}
for a.e. $(x,v)\in X\times V$. 
Therefore performing the change of variables $z=x-tv-t_1v_1$ ($dz=t_1^{n-1}dt_1dv_1$) on the right hand side of \eqref{t3}, we obtain
\begin{eqnarray}
&&\beta_3(x,v,x',v')=\int_0^{\tau_-(x,v)}\int_X{E(x, x-tv,z,x')k(x',v',\widehat{z-x'})\over |x-tv-z|^{n-1}|z-x'|^{n-1}}\nonumber\\
&&\times k(x-tv,\widehat{x-tv-z},v)k(z,\widehat{z-x'},\widehat{x-tv-z})dtdz,\label{t5}
\end{eqnarray}
for a.e. $(x,v,x',v')\in (X\times V)^2$.
Then by induction we have
\begin{eqnarray}
\hspace{-2cm}
&&\beta_m(x,v,z_m,v_m)=\int_0^{\tau_-(x,v)}\int_{X^{m-2}}{E(x,x-tv,z_2,\ldots, z_m)\over |x-tv-z_2|^{n-1}\Pi_{i=2}^{m-1}|z_i-z_{i+1}|^{n-1}}\label{t6b}\\
\hspace{-2cm}
&&\times k(x-tv,v_1, v)\Pi_{i=2}^mk(z_i,v_i,v_{i-1})_{|v_1=\widehat{x-tv-z_2},\ v_i=\widehat{z_i-z_{i+1}},\ i=2\ldots m-1}\nonumber
  dtdz_2\ldots dz_{m-1},\nonumber
\end{eqnarray}
for a.e. $(x,v,x',v')\in  (X\times V)^2$. This is the right hand side of \eqref{t6}. And we obtain \eqref{t6} for $\gamma_m$, $m\ge 3$, the distributional kernel of $j_+K^m.$

Now let $\psi\in L^1(\Gamma_-,d\xi)$. Then for $m\ge 3$ we have
\begin{eqnarray}
&&K^mJ\psi(z_0,v_0)=\int_{X\times V}\beta_m(z_0,v_0,z_m,v_m)J\psi(z_m,v_m)dz_m dv_m\nonumber\\
&=&\int_{X\times V}\int_0^{\tau_-(z_0,v_0)}\int_{X^{m-2}}\left[{E(z_0,\ldots, z_m)\over \Pi_{i=1}^{m-1}|z_i-z_{i+1}|^{n-1}}\right.\nonumber\\
\hspace{-2cm}
&&\times \left. \Pi_{i=1}^mk(z_i,v_i,v_{i-1})\right]_{|z_1=z_0-tv_0\ v_i=\widehat{z_i-z_{i+1}},\ i=1\ldots m-1}
  dtdz_2\ldots dz_{m-1}\nonumber\\
&&\times E(z_m,z_m-\tau_-(z_m,v_m)v_m)\psi(z_m-\tau_-(z_m,v_m)v_m,v_m)dz_m dv_m\nonumber\\
&=&\int_{\Gamma_-}\int_0^{\tau_+(z_{m+1},v_m)}\int_0^{\tau_-(z_0,v_0)}\int_{X^{m-2}}\left[{E(z_0,\ldots, z_m)\over \Pi_{i=1}^{m-1}|z_i-z_{i+1}|^{n-1}}\right.\label{t7d}\\
\hspace{-2cm}
&&\times \Pi_{i=1}^mk(z_i,v_i,v_{i-1})E(z_m,z_{m+1})\nonumber\\
&&\left.\times \psi(z_{m+1},v_m)\right]_{|z_1=z_0-tv_0,\ z_m=z_{m+1}+sv_m,\atop v_i=\widehat{z_i-z_{i+1}},\ i=1\ldots m}\nonumber
  dsdz_2\ldots dz_{m-1}d\xi(z_{m+1},v_m), 
\end{eqnarray}
for a.e. $(z_0,v_0)\in X\times V$. (We used above the change of variables ``$(z_m,v_m)=(z_{m+1}+tv_m,v_m)$", $(z_{m+1},v_m)\in \Gamma_-$, $0< t<\tau_+(z_{m+1},v_m)$, $dtd\xi(z_{m+1},v_m)=dz_m dv_m$.) This yields \eqref{t7}. For $m=2$, we have
\begin{eqnarray}
&&K^2J\psi(z_0,v_0)=\int_{X\times V}\beta_2(z_0,v_0,z_2,v_2)J\psi(z_2,v_2)dz_2 dv_2\nonumber\\
&=&\int_{X\times V}\int_0^{\tau_-(z_0,v_0)}\!\!\!\!\!\!\!E(z_0,z_0-tv_0,z_2){k(z_0-tv_0,v_1,v_0)k(z_2,v_2,v_1)_{|v_1= \widehat{z_0-tv_0-z_2}}\over |z_0-tv_0-z_2|^{n-1}}dt\nonumber\\
&&\times E(z_2,z_2-\tau_-(z_2,v_2)v_2)\psi(z_2-\tau_-(z_2,v_2)v_2,v_2)dz_2 dv_2\nonumber\\
&=&\int_{\Gamma_-}\psi(z_3,v_2)\int_0^{\tau_+(z_3,v_2)}\int_0^{\tau_-(z_0,v_0)}\!\!\!\!\!\!\!E(z_0,z_0-tv_0,z_2,z_3)\nonumber\\
&&\times \Big({k(z_0-tv_0,v_1,v_0)k(z_2,v_2,v_1)\over |z_0-tv_0-z_2|^{n-1}}
\Big)_{\big|{z_2=z_3+sv_2\atop 
v_1= \widehat{z_0-tv_0-z_2}}}dt ds d\xi(z_3,v_2),
\end{eqnarray}
and we obtain \eqref{t7c}. 

Formula \eqref{t7b} is just a rewriting of the definition of the operator $KJ$ from \eqref{eq:Isteady} and \eqref{eq:K}.
\end{proof}

\begin{proof}[Proof of Lemma \ref{est2}]

We have (see \eqref{t6}) 
\begin{eqnarray}
&&|\gamma_{n+2}(z_0,v_0,z_{n+2},v_{n+2})|\le e^{(n+2)\|\tau\sigma_-\|_\infty}\|k\|_\infty^{n+2}\nonumber\\
&&\times
\int_0^{\tau_-(z_0,v_0)}\int_{X^n}\left[{1\over \Pi_{i=1}^{n+1}|z_i-z_{i+1}|^{n-1}}\right]_{|z_1=z_0-tv_0}
  dtdz_2\ldots dz_{n+1},\label{f7}
\end{eqnarray}
for a.e. $(z_0,v_0,z_{n+2},v_{n+2})\in \Gamma_+\times X\times V$.

First assume that $n=2$. Then from \eqref{f1}, it follows that there exist positive constants $C_1$ and $C_2$ such that
\begin{equation*}
\int_{X^2}{dz_2 dz_3\over \Pi_{i=1}^3|z_i-z_{i+1}|^{n-1}}
  \le C_1\int_X{(C_2-\ln|z_1-z_3|)dz_3\over |z_3-z_4|},
\end{equation*}
for any $(z_1,z_4)\in \bar X \times\bar X$, $z_1\not=z_4$. Then from \eqref{f3}  and \eqref{f7}, we obtain \eqref{e1} when $n=2$.

Now assume that $n\ge 3$.
Then from \eqref{f2}, it follows inductively that
\begin{equation}
\int_{X^n}{dz_2\ldots dz_{n+1}\over \Pi_{i=1}^{n+1}|z_i-z_{i+1}|^{n-1}}
  \le C\int_{X^{n-l+2}}{ dz_l\ldots dz_{n+1}\over |z_1-z_l|^{n-l+1}\Pi_{i=l}^{n+1}|z_i-z_{i+1}|^{n-1}}
\end{equation}
for $3\le l\le n$ and for  any $(z_1, z_{n+2})\in \bar X\times\bar X$, $z_1\not=z_{n+2}$ and for some constant $C$. With $l=n$ above and using \eqref{f1} and \eqref{f3}, we obtain
\begin{eqnarray}
\int_{X^n}{dz_2\ldots dz_{n+1}\over \Pi_{i=1}^{n+1}|z_i-z_{i+1}|^{n-1}}
  &\le &C\int_{X^2}{dz_n dz_{n+1}\over |z_1-z_n||z_n-z_{n+1}|^{n-1}|z_{n+1}-z_{n+2}|^{n-1}}\nonumber\\
&\le&  C\int_X{(C'-\ln|z_1-z_{n+1}|)dz_{n+1}\over |z_{n+1}-z_{n+2}|^{n-1}}\le C'',\label{f8}
\end{eqnarray}
for $(z_1,z_{n+2})\in \bar X^2$, $z_1\not=z_{n+2}$, for some positive constants $C'$ and $C''$. Combining \eqref{f7} and \eqref{f8} yields property \eqref{e1} when $n\ge 3$.

Estimate \eqref{e4} follows from \eqref{t7c} for $m=2$. Assume that $m \ge 3$.
From \eqref{t7} it follows that
\begin{eqnarray}
&&|\alpha_m(z_0,v_0,z_{m+1},v_m)|\le\|k\|_\infty^m e^{(m+1)\|\tau \sigma_-\|_\infty}|\nu(z_{m+1})\cdot v_m|\label{f9}\\
&&\times\int_0^{\tau_+(z_{m+1},v_m)}\int_0^{\tau_-(z_0,v_0)}\int_{X^{m-2}}\left[{1\over \Pi_{i=1}^{m-1}|z_i-z_{i+1}|^{n-1}}\right]_{|z_1=z_0-tv_0,\ z_m=z_{m+1}+sv_m}
 \hspace{-3cm} ds dt dz_2\ldots dz_{m-1},\nonumber
\end{eqnarray}
for a.e. $(z_0,v_0,z_{m+1},v_m)\in \Gamma_+\times \Gamma_-$.

First, we consider the case when $n\ge 3$. We proceed as before by induction with the help of  \eqref{f2} to obtain
\begin{eqnarray}
&&\int_{X^{m-2}}{dz_2\ldots dz_{m-1}\over \Pi_{i=1}^{m-1}|z_i-z_{i+1}|^{n-1}}
 \le C_l\int_{X^{m-l}}{dz_l\ldots dz_{m-1}\over|z_1-z_l|^{n-l+1}\Pi_{i=l}^{m-1}|z_i-z_{i+1}|^{n-1}}\nonumber\\
 &\le&C_{m-1}\int_{X}{dz_{m-1}\over |z_1-z_{m-1}|^{n-m+2}|z_{m-1}-z_m|^{n-1}}\label{f11}\\
 &\le&{C_m\over |z_1-z_m|^{n-m+1}}\label{f10}
\end{eqnarray}
for some positive constant $C_j$, $j=3\ldots m\le n$. Hence combining  \eqref{f9} and \eqref{f10}, we obtain \eqref{e4} when $n\ge 3$.

Assume again that $n\ge 3$ and set $m=n+1$ in \eqref{f11}. From \eqref{f1} and \eqref{f9} it follows that there exist positive constants $C$ and $C'$ so that
\begin{eqnarray}
&&|\alpha_{n+1}(z_0,v_0,z_{n+2},v_{n+1})|\le e^{(n+2)\|\tau\sigma_-\|_\infty}\|k\|_\infty^{n+1}|\nu(z_{n+2})\cdot v_{n+1}|\nonumber\\
&&\hskip-1cm\times\int_0^{\tau_+(z_{n+2},v_{n+1})}\!\!\!\!\!\int_0^{\tau_-(z_0,v_0)}\!\!\!\!\!\big(C-C'\ln(|z_0-tv_0-z_{n+2}-sv_{n+1}|)ds dt.\label{f12}
\end{eqnarray}
Estimate \eqref{f12} also holds when $n=2$ (see \eqref{f1} and \eqref{f9}).
Then there exists a positive constant $C_0$ that depends only on $\diam$ such that
\begin{equation}
\sup_{(w,\theta)\in \R^n\times V\atop |w|\le 2\diam}\int_0^{\diam}|\ln(|w-s\theta|)|ds\le C_0.
\end{equation}
We apply the above inequality for ``$(w,\theta)=(z_0-z_{n+2}-tv_0, v_{n+1})$" on \eqref{f12} and obtain 
\begin{eqnarray*}
|\alpha_{n+1}(z_0,v_0,z_{n+2},v_{n+1})|&\le& e^{(n+2)\|\tau\sigma_-\|_\infty}\|k\|_\infty^{n+1}|\nu(z_{n+2})\cdot v_{n+1}|\\
&&\times\tau_-(z_0,v_0)\big(C\diam+C_0),
\end{eqnarray*}
for a.e. $(z_0,v_0,z_{n+2},v_{n+1})\in \Gamma_+\times \Gamma_-$. Property \eqref{e2} follows.
\end{proof}

We will use the following Lemma, consisting of basic estimates, in the proof of Propositions \ref{ballistic_test} and \ref{singlescat_test}.
\begin{lemma}
The following estimates hold: Let $D\in (0,+\infty)$, then
\begin{equation}
\sup_{w''\in \R^{n-1}\atop |w''|\le D}\int_{w\in \R^{n-1}\atop |w|\le \tilde \eta} \int_0^{2D}{dtdw\over |w+w''|+|t|}\le C(D)\tilde \eta^{n-1}(1+|\ln(\tilde \eta)|),
\label{P1}
\end{equation}
for $\tilde \eta\in (0,D)$ and for some constant $C(D)$ that depends on $D$.
In addition
\begin{equation}
\sup_{w''\in \R^{n-1}}\int_{w\in \R^{n-1}\atop |w|\le \tilde \eta} \int_0^{+\infty}{dt dw\over (|w+w''|+|t|)^{n-m+1}}\le C_m\tilde \eta^{m-1},\label{P2}
\end{equation}
for $\tilde \eta\in (0,+\infty)$ and for $2\le m\le n-1$ and some constant $C_m$.
\end{lemma}

\begin{proof}[Proof of Proposition \ref{ballistic_test}]
We prove \eqref{i5}. From \eqref{e1}, it follows that
\begin{eqnarray}
&&\Big|\int_{\Gamma_+}\phi(x,v)j_+K^{n+2}(I-K)^{-1}Jf(x,v)d\xi(x,v)\Big|\nonumber\\
&&\le \int_{\Gamma_+}\tau_-(x,v)|\phi(x,v)|d\xi(x,v)\nonumber\\
&&\hskip -0.5cm\times\Big\|\int_{X\times V}\!\!\!\!\!{\gamma_{n+2}(z_0,v_0,z_{n+2},v_{n+2})\over\tau_-(z_0,v_0)}(I-K)^{-1}Jf(z_{n+2},v_{n+2})
dz_{n+2}dv_{n+2}\Big\|_{L^\infty({\Gamma_+}_{z_0,v_0})}\nonumber\\
&\le&  Ce^{(n+2)\|\tau\sigma_-\|_\infty}\|k\|_{\infty}^{n+2}\|(I-K)^{-1}J\|_{\mathcal{L}(L^1(\Gamma_-,d\xi),L^1(X\times V,\tau^{-1}dx dv))}\nonumber\\
&&\times\int_{\Gamma_+}\tau_-(x,v)|\phi(x,v)|d\xi(x,v)\label{f30}
\end{eqnarray}
for $f\in L^1(\Gamma_-,d\xi)$, $\|f\|_{L^1(\Gamma_-,d\xi)}=1$, and some constant $C$.
Then since ${\rm supp}\phi\subset\{(z,v)\in \Gamma_+\ |\ |z-z'-\tau_+(z',v')v'|\le \eta,\ |v-v'|\le \eta\}$ and $\|\phi\|_\infty\le 1$ and ${\rm dim}\Gamma_+=2n-2$, we have
\begin{equation}
\int_{\Gamma_+}\tau_-(z,v)|\phi(z,v)|d\xi(z,v)\le \diam\int_{{\rm supp}\phi}d\xi(z_{n+2},v_{n+2})\le C\eta^{2n-2},\label{f31}
\end{equation}
for some constant $C$. Combining \eqref{f30} and \eqref{f31}, we obtain \eqref{i5}.

We now prove \eqref{i4}. From \eqref{e2}, it follows that 
\begin{eqnarray}
&&{1\over |\nu(z'')\cdot v''|}\big|\int_{\Gamma_+}\alpha_{n+1}(z,v,z'',v'')\phi(z,v)d\xi(z,v)\big|\nonumber\\
&\le& Ce^{(n+2)\|\tau\sigma_-\|_\infty}\|k\|_{\infty}^{n+1}\int_{\Gamma_+}\tau_-(z,v)|\phi(z,v)|d\xi(z,v),\label{f32}
\end{eqnarray}
for a.e. $(z'',v'')\in \Gamma_-$ and for some constant $C$ which depends on the diameter of $X$. Then using again \eqref{f31}, we obtain \eqref{i4}.

We next prove \eqref{i6}. Let $f\in L^1(\Gamma_-,d\xi)$, $\|f\|_{L^1(\Gamma_-,d\xi)}=1$. We start from \eqref{t7b} and obtain
\begin{eqnarray*}
&&\Big|\int_{\Gamma_+}\phi(z,v)j_+KJf(z,v)d\xi(z,v)\Big|\le e^{2\|\tau\sigma_-\|_\infty}\|k\|_\infty \\
&&\times \int_{\Gamma_+\times V_{v_1}}|\phi(z,v)|\int_0^{\tau_-(z,v)}|f|(z-tv-\tau_-(z-tv,v_1)v_1,v_1)dt dv_1d\xi(z,v).
\end{eqnarray*}
Then, performing the change of variables $``(x,v)=(z-tv,v)"$ and again the change of variables $``(z_1+sv_1,v_1)=(x,v_1)"$, we have:
\begin{eqnarray*}
&&\Big|\int_{\Gamma_+}\phi(z,v)j_+KJf(z,v)d\xi(z,v)\Big|\\
&\le &e^{2\|\tau\sigma_-\|_\infty}\|k\|_\infty \int_{X\times V_v\times V_{v_1}}|\phi(x+\tau_+(x,v)v,v)||f|(x-\tau_-(x,v_1)v_1,v_1) dv_1 dx dv\\
&\le &e^{2\|\tau\sigma_-\|_\infty}\|k\|_\infty \int_{\Gamma_-\times V_v}\int_0^{\tau_+(z_1,v_1)}|\phi(z_1+sv_1+\tau_+(z_1+sv_1,v)v,v)|ds\\
&&\times|f|(z_1,v_1) d\xi(z_1,v_1)dv.
\end{eqnarray*}
Since ${\rm supp}\phi\subset\{(z,v)\in \Gamma_+\ |\  |v-v'|\le \eta\}$, we obtain
\begin{eqnarray*}
&&\Big|\int_{\Gamma_+}\phi(z,v)j_+KJf(z,v)d\xi(z,v)\Big|\\
&\le &e^{2\|\tau\sigma_-\|_\infty}\|k\|_\infty \|f\|_{L^1(\Gamma_-,d\xi)}\diam \int_{v\in V\atop |v'-v|\le \eta}dv\le Ce^{2\|\tau\sigma_-\|_\infty}\|k\|_\infty\eta^{n-1},
\end{eqnarray*}
for some positive contant $C$, which proves \eqref{i6}.

We now prove \eqref{i1} and \eqref{i2}.  Let $2\le m\le n$. From \eqref{e4}, it follows that for a.e. $(z'',v'')\in \Gamma_-$
\begin{eqnarray}
&&{\Big|\int_{\Gamma_+}\alpha_m(z,v,z'',v'')\phi(z,v)d\xi(z,v)\Big|\over |\nu(z'')\cdot v''|}\le Ce^{(m+1)\|\tau\sigma_-\|_\infty}\|k\|_{\infty}^m\nonumber\\
&& \hskip-1cm\times\int_{\Gamma_+}|\phi|(z,v)\int_0^{\tau_+(z'',v'')}\int_0^{\tau_-(z,v)}{dt ds\over |z''+tv''-z+sv|^{n-m+1}}d\xi(z,v).\label{L1}
\end{eqnarray}
By the change of variable $``(x,v)=(z-sv,v)"$, we have
\begin{eqnarray}
&&\int_{\Gamma_+}|\phi|(z,v)\int_0^{\tau_+(z'',v'')}\int_0^{\tau_-(z,v)}{dt ds\over |z''+tv''-z+sv|^{n-m+1}}d\xi(z,v)\nonumber\\
&=&\int_{X\times V}|\phi|(x+\tau_+(x,v)v,v)\int_0^{\tau_+(z'',v'')}{dt \over |z''+tv''-x|^{n-m+1}}dx dv\nonumber\\
&\le &\int_0^{\tau_+(z'',v'')}\Big(\int_{U_\eta}{dx dv \over |x''-x|^{n-m+1}}\Big)_{|x''=z''+tv''}dt.\label{L2}
\end{eqnarray}
for  $(z'',v'')\in \Gamma_-$, where 
$$
U_\eta:=\{(x,v)\in X\times V\ |\  \max(|v-v'|,\ |x+\tau_+(x,v)v-z'-\tau_+(z',v')v'|)\le \eta\}.
$$
For $(x,v)\in U_\eta$ then 
\begin{eqnarray}
&&|x-z'+(\tau_+(x,v)-\tau_+(z',v'))v'|\le |x+\tau_+(x,v)v-z'-\tau_+(z',v')v'|\nonumber\\
&&\hskip 4cm+|v-v'|\tau_+(x,v)
\le \eta(1+\diam).\label{L3b}
\end{eqnarray}
Therefore, 
\begin{equation}
U_\eta\subseteq\{(z'+w+\lambda v',v)\ | \ |v-v'|\le \eta,\ |w|\le\tilde \eta,\ w\cdot v'=0,\ \lambda\in (-\diam,\diam)\},\label{L3}
\end{equation}
where 
\begin{equation}
\tilde \eta=\min(\eta(1+\diam),\diam).\label{L5}
\end{equation}
We can also write any $x''\in X$ as a sum $z'+w''+\lambda'' v'$ where $|w''|\le\diam$, $w''\cdot v'=0$, $\lambda\in (-\diam,\diam)$. Therefore from  
\eqref{L3} and \eqref{P1} (with ``$D=\diam$") and \eqref{P2} it follows that 
\begin{eqnarray}
&&\sup_{x''\in X}\int_{U_\eta}{dx dv \over |x''-x|^{n-m+1}}\nonumber\\
&\le&2^{n-m+3\over 2}
\sup_{w''\in \R^{n-1}\atop |w''|\le \diam}\int_{v\in V\atop|v-v'|\le \eta}\int_{w\in \R^{n-1}\atop |w|\le \tilde \eta} \int_0^{2\diam}{ds dw\over (|w+w''|+|s|)^{n-m+1}}\nonumber\\
&\le& C\left\lbrace 
\begin{array}{l}
\tilde \eta^{2n-2}(1+|\ln(\tilde \eta)|),\textrm{ when }m=n,\\
\tilde \eta^{n+m-2},\textrm{ when }m<n,
\end{array}
\right.
\label{L5b}
\end{eqnarray}
for some constant $C$.
Combining \eqref{L2}, \eqref{L5b} and \eqref{L5}, we obtain that
\begin{eqnarray}
&&\int_{\Gamma_+}|\phi|(z,v)\int_0^{\tau_+(z'',v'')}\int_0^{\tau_-(z,v)}{dt ds\over |z''+tv''-z+sv|^{n-m+1}}d\xi(z,v)\nonumber\\
&\le &C'
 \left\lbrace
 \begin{array}{l}
\eta^{2n-2}(1+|\ln( \eta)|),\textrm{ when }m=n,\\
\eta^{n+m-2},\textrm{ when }m<n,
\end{array}
\right.\label{L7}
\end{eqnarray} 
for $(z'',v'')\in \Gamma_-$ and some positive constant $C'$. Combining \eqref{L1} and \eqref{L7} we obtain \eqref{i1} and \eqref{i2}.
\end{proof}

\begin{proof}[Proof of Proposition \ref{singlescat_test}]
For $(x,v)\in X\times V$ we still denote by $d(x,v)$ the distance between the lines $x+\R v$ and $z'+\R v'$, $d(x,v):=\inf_{s,t\in \R} |x+sv-z'-tv'|$.

We first prove \eqref{j4} and \eqref{j3}.
Since ${\rm supp}\phi\subset\{(z,v)\in \Gamma_+\ |\ d(z,v)\le \eta\}$ and $\|\phi\|_\infty\le 1$, we have
\begin{eqnarray}
\int_{\Gamma_+}\tau_-(z,v)|\phi(z,v)| d \xi(z,v)
&=&\int_{X\times V}|\phi(x+\tau_+(x,v)v,v)|dx dv\nonumber\\
&\le& \int_{(x,v)\in X\times V,\ d(x,v)\le \eta }\hskip-0.5cm dx dv\le C\eta^{n-2},\label{j10}
\end{eqnarray}
for some constant $C$ that depends on $X$, and from \eqref{f30} it follows that
\begin{equation*}
\Big|\int_{\Gamma_+}\phi(x,v)j_+K^{n+2}(I-K)^{-1}Jf(x,v) d\xi(x,v)\Big|
\le C'e^{(n+2)\|\tau\sigma_-\|_\infty}\|k\|_{\infty}^{n+2}\eta^{n-2},
\end{equation*}
for $f\in L^1(\Gamma_-,d\xi)$, $\|f\|_{L^1(\Gamma_-,d\xi)}=1$ and some constant $C'$, which proves \eqref{j4}.  Similarly  \eqref{j3} follows from  \eqref{f32} and \eqref{j10}.

We make the following remark: For $(x,v)\in X\times V$ so that $d(x,v):=\inf_{s,t\in \R} |x+sv-z'-tv'|\le \eta$ and $v\not=\pm v'$, we have 
\begin{eqnarray}
&&x-z'=w+\lambda_1 v'+\lambda_2{v-(v\cdot v') v'\over |v-(v\cdot v') v'|},\label{j5}\\
&&w\cdot v'=w\cdot v=0,\ |w|\le\tilde \eta,\ \max(|\lambda_1|,|\lambda_2|)\le \diam,\label{j6}
\end{eqnarray}
where 
\begin{equation}
\tilde \eta=\min(\eta, \diam) \label{j6b}
\end{equation} 
(the vector $w$ can be identified with a vector of the $n-2$ dimensional ball  of center 0 and radius $\tilde \eta$). 

We now prove \eqref{j1} and \eqref{j1b}. Let $n\ge 3$ and $2\le m\le n-1$. From \eqref{e4}, it follows that
\begin{eqnarray}
&&{\Big|\int_{\Gamma_+}\alpha_m(z,v,z'',v'')\phi(z,v)d\xi(z,v)\Big|\over|\nu(z'')\cdot v''|}
\le Ce^{(m+1)\|\tau\sigma_-\|_\infty}\|k\|_{\infty}^m\nonumber\\
&&\hskip -1cm\times\int_{X\times  V}|\phi|(x+\tau_+(x,v)v,v)\int_0^{\tau_+(z'',v'')}{dt \over |z''+tv''-x|^{n-m+1}}dx dv\label{j7}
\end{eqnarray}
for a.e. $(z'',v'')\in \Gamma_-$ (we proceed as we did for \eqref{L1} and \eqref{L2}). By \eqref{j5} and \eqref{j6} (where we also write ``$z''-z'=w''+\lambda_1''  v'+\lambda_2''{v-(v\cdot v') v'\over |v-(v\cdot v') v'|}$", $|w''|\le \diam$") we have
\begin{eqnarray}
&&\int_{\Gamma_+}|\phi|(z,v)\int_0^{\tau_+(z'',v'')}\int_0^{\tau_-(z,v)}{dt ds\over |z''+tv''-z+sv|^{n-m+1}}d\xi(z,v)\nonumber\\
&\le&C\sup_{w''\in \R^{n-2}\atop |w''|\le\diam}\int_{w\in \R^{n-2}\atop |w|\le \tilde \eta}\int_0^{2\diam}\int_0^{2\diam} 
{dw d\lambda_1 d\lambda_2\over (|w+w''|+|\lambda_1|+|\lambda_2|)^{n-m+1}}\nonumber\\
&\le& {C\over n-m}\sup_{w''\in \R^{n-2}\atop |w''|\le\diam}\int_{w\in \R^{n-2}\atop |w|\le \tilde \eta}\int_0^{2\diam} 
\hskip -8mm{dw d\lambda_1 \over (|w+w''|+|\lambda_1|)^{n-m}},\label{j8}
\end{eqnarray}
for some constant $C$.
Then we use \eqref{P1} or \eqref{P2} (in dimension $n-1$ instead of dimension $n$) and obtain
\begin{eqnarray}
&&\int_{\Gamma_+}|\phi|(z,v)\int_0^{\tau_+(z',v')}\int_0^{\tau_-(z,v)}{dt ds\over |z''+tv''-z+sv|^{n-m+1}}d\xi(z,v)\nonumber\\
&\le&\left\lbrace
\begin{array}{l}
C\eta^{m-1},\textrm{ when }m\le n-2,\\
C\eta^{n-2}(1+|\ln(\eta)|)\textrm{ when }m= n-1,
\end{array}
\right.\label{j9}
\end{eqnarray}
for some constant $C$. Combining \eqref{j7}--\eqref{j9}, we obtain \eqref{j1} and \eqref{j1b}.

We finally prove \eqref{j2}. Note \eqref{j7} also holds for $m=n$. Then
as above, we obtain that there exists a positive constant $C$ so that
\begin{eqnarray}
&&\int_{\Gamma_+}|\phi|(z,v)\int_0^{\tau_+(z'',v'')}\int_0^{\tau_-(z,v)}{dt ds\over |z''+tv''-z+sv|}d\xi(z,v)\nonumber\\
&\le &C\sup_{w''\in \R^{n-2}\atop |w''|\le\diam}\int_{w\in \R^{n-2}\atop |w|\le \tilde \eta}\int_0^{2\diam}\int_0^{2\diam} 
{dw d\lambda_1 d\lambda_2\over (|w+w''|+|\lambda_1|+|\lambda_2|)}\nonumber\\
&\le& C'\sup_{w''\in \R^{n-2}\atop |w''|\le\diam}\int_{w\in \R^{n-2}\atop |w|\le \tilde \eta}\int_0^{2\diam}
(1+|\ln(|w+w''|+|\lambda_1|)|)d\lambda_1 dw\nonumber\\
&\le&C''\sup_{w''\in \R^{n-2}\atop |w''|\le\diam}\int_{w\in \R^{n-2}\atop |w|\le \tilde \eta}dw
\le C'''\eta^{n-2},
\end{eqnarray}
for a.e. $(z'',v'')\in \Gamma_-$ and for some positive constant $C$, $C'$, $C''$ and $C'''$ that depend on the diameter of $X$.  Combining \eqref{j7} for $m=n$ and this latter estimate we obtain \eqref{j2}, which concludes the proof of the proposition.
\end{proof}

\section*{Acknowledgment} AJ is partially supported by French grant ANR-13-JS01-0006. GB would like to thank the Universit\'e de Lille 1, where part of this work was completed, for their hospitality. GB also acknowledges partial support from the National Science Foundation and the Office of Naval Research.


\end{document}